\documentclass[]{amsart}

\usepackage[utf8]{inputenc}
\usepackage[OT2,T1]{fontenc}
\DeclareSymbolFont{cyrletters}{OT2}{wncyr}{m}{n}
\DeclareMathSymbol{\Sha}{\mathalpha}{cyrletters}{"58}
\usepackage{graphicx}
\graphicspath{ {./images/} }

\usepackage[hyphens,spaces,obeyspaces]{url}
\usepackage[colorlinks,allcolors=blue,hyperindex,breaklinks]{hyperref}
\hypersetup{
           breaklinks=true,   
           colorlinks=true,   
        }

\usepackage{orcidlink}

\usepackage{pgfplots}
\pgfplotsset{compat=1.18}

\title[Whitney extensions on symmetric spaces]{Whitney extensions on symmetric spaces}
\date{\today}
\usepackage[foot]{amsaddr}

\author[Birgit Speh]{Birgit Speh \orcidlink{0000-0003-0869-306X}}
\address{Department of Mathematics, Cornell University, Ithaca, NY 14853-4201, USA}
\email{speh@math.cornell.edu}

\author[Peter Vang Uttenthal]{Peter Vang Uttenthal \orcidlink{0009-0001-0878-8213}}
\address{Department of Mathematics, Aarhus University, 8000 Aarhus C, Denmark}
\email{petervang@math.au.dk}



\newcommand{\Q}{\mathbb{Q}}
\newcommand{\Z}{\mathbb{Z}}

\newcommand{\R}{\mathbb{R}}
\newcommand{\cc}{\mathbb{C}}

\usepackage{amsthm,amsmath, mathrsfs, mathtools}
\usepackage{amsfonts}
\usepackage{amssymb}
\usepackage{fancyhdr}
\usepackage{IEEEtrantools}
\usepackage{tikz-cd}
\usepackage[english]{babel}
\usepackage[utf8]{inputenc}
\usepackage{csquotes,comment}

\newtheorem{theorem}{Theorem}
\numberwithin{theorem}{section}
\newtheorem{lemma}{Lemma}
\numberwithin{lemma}{section}

\numberwithin{definition}{section}
\newtheorem{proposition}{Proposition}
\numberwithin{proposition}{section}
\newtheorem{corollary}{Corollary}
\numberwithin{corollary}{section}
\newtheorem{remark}{Remark}
\numberwithin{remark}{section}

\numberwithin{conjecture}{section}

\numberwithin{question}{section}

\numberwithin{example}{section}

\usepackage[hyphens,spaces,obeyspaces]{url}
\usepackage[colorlinks,allcolors=blue,hyperindex,breaklinks]{hyperref}

\hypersetup{
           breaklinks=true,   
           colorlinks=true,   
        }
\usepackage{tabularx}
\usepackage{subcaption}
\begin{document}

\begin{abstract} In 1934, H. Whitney introduced the problem of extending a function $f$ on a set of points in $\mathbb{R}^n$ to an analytic function on the ambient space. In this article we prove Whitney type extension theorems for data on some homogeneous spaces.  We use harmonic analysis on the homogeneous spaces and representation theory of compact as well as noncompact reductive groups. 
\end{abstract}

\maketitle

\tableofcontents


\section{Introduction}
In 1934, H. Whitney introduced the problem of extending a function $f$ on a set of points in $\R^n$ to an analytic function on the ambient space.  In the 2000s, C. Fefferman took up Whitney's work and proved a sharp form of the Whitney extension theorem \cite{Fefferman}. In particular, if the set of points in Euclidean space is finite of cardinality $n,$ there is an extension which is continuously differentiable of some degree depending on $n.$

In this paper, we study a version of Whitney's extension problem where the points have an inherent symmetry.  
More precisely, we consider a set of points 
$\mathscr{D} = \{x^{(i)} \in G/H: 1 \leqslant i \leqslant n\}$ in a symmetric homogeneous  space $X=G/H$ for reductive Lie groups $G$ and $H$ and a nonconstant function $f: \mathscr{D} \to \mathbb{C}$. We are especially interested not just in proving the existence of an analytic function on $G/H$ extending $f$, 
but to obtain an explicit formula for this function which can be effectively computed by an algorithm. 
The analytic extension, which we will denote $W_{f,\mathscr{D}}$, generates a representation $\Pi_n$ of $G$, which we call the Whitney representation of the data. 
In the main part of the article we consider symmetric spaces  of the form $G/H$ for certain orthogonal  Lie groups $G$ and $H$, but the results can  be modified to other symmetric spaces.

The analysis on symmetric  spaces  was intensively studied starting with the work of M. Flensted-Jensen in the late 1970s, leading to the Plancherel theorem around 2005 by H. Schlichtkrull and E. van den Ban \cite{Schlichtkrull-vandenBan1, Schlichtkrull-vandenBan2}. In this article, we use  information about the representation of the symmetry group $G$ on the square integrable functions on $G/H$ and in particular about the irreducible subrepresentations. We also use results about the restriction of these representations to orthogonal subgroups.

Let $p, q \geqslant 1$ be integers.
Let $(\cdot, \cdot)_{p,q}$ be the symmetric bilinear form on $\R^{p+q}$ 
of signature $(p,q)$
defined by
$$
(x,y)_{p,q} = \sum_{k=1}^p x_k y_k - \sum_{k=p+1}^{p+q} x_k y_k, \quad x,y \in \R^{p+q}.
$$
The indefinite orthogonal group of signature $ (p,q)$ is the group
$$
\operatorname{O}(p,q) = \{ g \in \operatorname{GL}(p+q,\R) : (gx,gy)_{p,q} = (x,y)_{p,q} \  \forall x,y\in \R^{p+q}\},
$$
and $\operatorname{SO}(p,q)= \{ g\in \operatorname{O}(p,q): \det g = 1 \}.$
For any Lie group $G$, the connected component of the identity is denoted $G^\circ$.
The connected noncompact Lie group $G = \operatorname{SO}(p,q)^\circ$ is the group of symmetries of the real hyperboloids 
\begin{center}
$X(p,q)_{\pm} = \{ (x_i)\in \R^{p+q} | x_1^2 + ... +x_p^2 - x_{p+1}^2 \ldots - x_{p+q}^2 = \pm 1 \} $
\end{center}
This class of homogeneous spaces contains the spaces known as 
\begin{center} 
\emph{de Sitter spaces} ($\operatorname{dS}$)
\end{center}
respectively as 
\begin{center}
\emph{anti de Sitter spaces} ($\operatorname{AdS}$)  
\end{center}
\cite[Chapters X and XI]{Zee}. 

\subsection{Main results}
We approach the extension problem by realizing the space $X(p,q)_{-}$ as the homogeneous space
$$
X(p,q)_{-} = \operatorname{SO}(p,q)^\circ/\operatorname{SO}(p,q-1)^\circ.
$$
This allows us to obtain an explicit formula for the Whitney function $f_n$ using the representation theory of noncompact Lie groups. 
In particular, we build on the work of M. Flensted-Jensen \cite{Flensted-Jensen}, who in 1980  published   an article in which he constructs a family of square integrable functions on symmetric spaces  which are eigenfunctions of all invariant differential operators.  We refer to these functions as Flensted-Jensen functions.

Before stating the theorem below, we note that hyperboloids can be parametrized with angular variables 
$(r,s) = (r_1, \ldots, r_{p},s_1,\dots,s_q) \in S^{p-1} \times S^{q-1}$ 
and a radial variable $t > 0$ 
such that all points in $X(p,q)_{-}$ are of the form
$$
x = (r_1\sinh(t), r_2 \sinh(t), \ldots, r_{p}\sinh(t), s_1\cosh(t), \ldots, s_q\cosh(t)).
$$
We prove
\begin{theorem} 
\label{whitney1} Let $p\geqslant 1$, $q\geqslant 2$, and $n \in \Z_{\geqslant 1}$.
Let $\mathscr{D} = \{x^{(i)} \in X(p,q)_- : 1\leqslant i \leqslant n \}$ and let $f: \mathscr{D} \to \mathbb{C}$ be nonconstant. For $\ell\in \Z$ and $k=\lfloor q/2 \rfloor$, define Flensted-Jensen functions 
\[
\varphi_\ell(x) =(s_1+\sqrt{-1}s_2+\cdots+s_{2k-1}+\sqrt{-1}s_{2k})^\ell\cosh(t)^{-\ell-q}, \quad x\in X(p,q)_-.
\]
If $\varphi_1|_{\mathscr{D}}$ is injective,
then there exists an analytic, square integrable function $W_{f, \mathscr{D}}: X(p,q)_{-} \to \mathbb{C}$ given at $x \in X(p,q)_-$ by the formula 
$$
W_{f, \mathscr{D}} (x) = \sum_{\ell \geqslant (p-q)/2} a_\ell (s_1 + \sqrt{-1}s_2+\cdots+s_{2k-1}+\sqrt{-1}s_{2k} )^{\ell} \cosh (t)^{-\ell-q}
$$
such that 
$W_{f, \mathscr{D}}(x)=f(x)$ for all $x\in \mathscr{D}$.
We refer to $W_{f, \mathscr{D}}$ as the Whitney extension of $f$ or as the Whitney function of the data $(f, \mathscr{D})$.
\end{theorem}

\medskip
In applications of the theorem to a list of observable data, we may assume
that the results of all measurements  are rational. In the case of the three-dimensional symmetric space $X(2,2)_{+} = \operatorname{SO}(2,2)^\circ/\operatorname{SO}(1,2)^\circ$, we prove a rational version of our Whitney extension theorem starting from the assumption that all the points are rational. 

\begin{theorem}
    Let $\mathscr{D} = \{ x^{(i)}\in X(2,2)_{+}(\Q): 1 \leqslant i \leqslant n\}$ be a list of $\Q$-valued points in the three-dimensional symmetric space $X(2,2)_{+}$ realized as the quadric surface 
    \[X(2,2)_{+} = \{x\in \R^4: x_1^2 + x_2^2 -x_3^2 - x_4^2 = 1 \}.\] 
with Flensted-Jensen functions 
\[
\varphi_\ell(x) = \left( x_1 + \sqrt{-1} x_2 \right)^{-\ell}, \quad x \in X(2,2)_{+}.
\]
for $\ell \in \Z_{\geqslant 1}$.    
Suppose $f: \mathscr{D}\to \Q$ is not constant and that  
$\varphi_1|_{\mathscr{D}}$ is injective. 
Then, we can choose coefficients $a_\ell$ in the number field $\Q(\sqrt{-1})$ such that 
    \[
W_{f,\mathscr{D}} = \sum_{2 \leqslant \ell \leqslant n+1} a_\ell \varphi_\ell
\]
is a Whitney extension of $f$, i.e. 
$W_{f,\mathscr{D}}: X(2,2)_+\to \mathbb{C}$ is analytic, square integrable, $K$-finite, and $W_{f,\mathscr{D}}(x) = f(x)$ for all $x\in \mathscr{D}$. 
\end{theorem}

Before proving 
Theorem \ref{whitney1}, we treat some lower dimensional examples in detail, since these are instructive for the more general cases. 
Furthermore, in Theorem \ref{finite-dim}, we prove Whitney extension theorems for functions on points in 
general symmetric spaces $G/H$ for classical groups $G$ and $H$ using tensor products of finite-dimensional representations. 
Finally, in Theorem \ref{G}, we give explicit formulas for Whitney extensions defined directly on noncompact Lie groups $G$ of nonexceptional type whenever $G$ admits holomorphic discrete series representations.

\subsection{Related work}
In the 2000s, C. Fefferman proved a sharp form of the Whitney extension theorem \cite{Fefferman}. In particular, if the set of points in Euclidean space is finite of cardinality $n$, Fefferman showed that there is an extension which is continously differentiable of some degree depending on $n$.  
In this paper we start with a set of points of finite cardinality $n$, and using additional assumptions we prove the existence of a Whitney function which is infinitely often differentiable. 
The reason we are able to get an extension of such high regularity is that we take advantage of the fact that the points live in a subset of Euclidean space obeying a high amount of symmetries.

Fefferman et al. \cite{FeffermanGeometric} have also studied Whitney extension problems in situations where the points live on a smooth manifold $X.$ The Whitney extension theorems proved in that context, however, are  existence theorems. In contrast, the results given in this article  are effective:  we give an algorithm computing the extension with an explicit formula. This is possible because we build the extension from functions in the discrete spectrum of  the symmetric space $G/H$. Such Flensted-Jensen functions are solutions to a system of  differential equations  and hence satisfy a high amount symmetries. 

Whitney extension theorems have found a range of applications in data science \cite{Fefferman-Data}. There is a growing literature on equivariant machine learning, where the data are invariant with respect to a group of symmetries \cite{Blum-Schmidt-Villar}. 
Most results in this direction, however, are only available when the symmetry group of the points is either finite or a \emph{compact} Lie group \cite{Blum-Schmidt-Villar}, \cite{Kondor-Trivedi}, \cite{Ennes-Tinarrage}.  
However, many groups of importance in physics are noncompact, such as the Lorentz group $\operatorname{SO}(3,1)^\circ$ and the
symmetry groups  $\operatorname{SO}(p,1)^\circ$ and $\operatorname{SO}(p,2)^\circ$ of the de Sitter and anti de Sitter spaces.   
This work was motivated by the attempt  to generalize the results in harmonic analysis applicable to Lie group equivariant machine learning to include \emph{noncompact} Lie groups.


The Whitney representations of the  orthogonal groups introduced in \ref{Whitneyreps} may be of interest in physics \cite{Varna}. For the sphere in the compact case, the representations are used to describe the hydrogen atom 
\cite{Hall-Quantum}. In the work of I. E. Segal the Whitney representations are used to describe electrons \cite{Segal-Elementary-particles}; see also S. Sternberg \cite{sternberg}. The proceedings of the biannual conferences in Varna (Bulgaria) organized by V. Dobrev discuss more recent developments in applications of symmetries to physics \cite{Varna}. 

\subsection{Notation}
If $X$ is a symmetric space, we will denote a set of points in $X$ of finite cardinality $n$ as
\[
\mathscr{D} = \{ x^{(i)} \in X: 1 \leqslant i \leqslant n \}
\]
where $x^{(i)} = (x^{(i)}_{1},\ldots,x_k^{(i)})$ if the ambient space of $X$ is $\R^k$.
The symbol $f$ will be used for functions on
$\mathscr{D}$ with values $y^{(i)} = f(x^{(i)}) \in \mathbb{C}$. 
If we wish to emphasize that $\mathscr{D}$ has cardinality $n$, 
we will write  
$f_n: \mathscr{D}_n \to \mathbb{C}$; on the other hand, if  
the set of points has been fixed throughout a statement, we will merely write $f: \mathscr{D} \to \cc$. 
Whenever a Whitney extension of $f: \mathscr{D}\to \cc$ has been constructed, it will be denoted as $W_{f,\mathscr{D}}$, while $W_{f, \mathscr{D}, n}$ will describe a Whitney extension of $f_n: \mathscr{D}_n \to \cc$. 
Additional notation will be introduced as needed. 

 \subsection{Open problems} 

\begin{enumerate}
\item Let $X$ be a symmetric space.
Consider a sequence of data sets 
$$\mathscr{D}_n =\{x^{(i)} \in X: \ 1\leqslant i\leqslant n\}$$ 
on which there are functions $f_n: \mathscr{D}_n \to \cc$ such that 
$\mathscr{D}_n \subset \mathscr{D}_{n+1}$ and $f_{n+1}|_{\mathscr{D}_n} = f_n$ for all $n \geq 1$. Does the sequence of Whitney functions $W_{f_n,\mathscr{D}_n}$ converge for $n \rightarrow \infty$? 

\item  Suppose that the first variables  $\{ ( x_1,\dots x_r )\}$ of the data are in a subvariety isomorphic to a symmetric space $X(p',q)$ of an orthogonal subgroup  $\operatorname{SO}(p',q)^\circ$ for $p' < p$.
How are the Whitney functions for $X(p',q)$ and $X(p,q)$ related? This can be viewed as a boundary value problem in partial differential equations and is closely related to the work of T. Kobayashi and M. Pevzner on holographic operators, cf. \cite{Pevzner}.

\item  Suppose that $U$ is an irreducible representation of a reductive subgroup $H$ of $G$. Consider the Whitney extension problem for the functions 
in   $$\operatorname{Ind}_H^G U .$$ In our present work we consider only the case of the trivial representation $U$.

\item  In this article, we limit our attention to symmetric spaces  but it might be of interest for applications to consider spherical spaces as well.
\end{enumerate}
\medskip

\subsection{Acknowledgements}
These results were first presented in a lecture at the AIM virtual seminar ``Representation Theory and Noncommutative Geometry''. We are grateful to Bent {\O}rsted for helpful conversations that greatly improved the final version of this paper. We would like to thank  David Vogan and Robert Yuncken for insightful questions and comments related to this work. We also want to thank the referee for pointing out some problems in the first version of the article and making very useful comments. 

B. Speh was supported by  grant 89086 from the Simons Foundation.
P. V. Uttenthal was supported by a research grant (VIL54509) from Villum Fonden.

\bigskip


\section{The Vandermonde trick} \label{vandermondetrick}

We will use the nonvanishing of Vandermonde determinants multiple times in this article.

\medskip
Let $\mathscr{D}=\{ x^{(i)} \in X: 1\leqslant i \leqslant n\}$  be a set of points in a smooth real manifold $X$ and consider a nonconstant 
function $f: \mathscr{D} \to \mathbb{C}$. 
Suppose 
$
\varphi: X\to \cc 
$
is an analytic function and let $k$ be any integer.
Let $(\varphi(x^{(i)})^\ell)_{i,\ell}$ be an  
$n \times n$ matrix where $1\leqslant i \leqslant n$ and $k\leqslant \ell \leqslant k+n -1$ and let $(f(x^{(i)}))_i$ be a vector in $\cc^n$.
For $(a_\ell)_\ell \in \cc^n$, consider the linear system
$$
(\varphi(x^{(i)})^{\ell})_{i,\ell} (a_\ell)_\ell = (f(x^{(i)}))_i.
$$
The matrix $(\varphi(x^{(i)})^{\ell})_{i,\ell}$
is a product of a diagonal matrix and a Vandermonde matrix (cf. \cite[p. 42]{Marcus}); hence
$$
\det(\varphi(x^{(i)})^{\ell})_{i,\ell} = 
\prod_{i \leqslant n} \varphi(x^{(i)})^k 
\prod_{i<j} \left( \varphi(x^{(i)}) - \varphi(x^{(j)}) \right).
$$
If $\varphi(x^{(i)}) \neq 0$ for all $i$ and if $
\varphi(x^{(i)}) \neq \varphi(x^{(j)})
$
whenever $i<j$, then there are $a_\ell \in \cc$ and an analytic function 
$$
W_{f,\mathscr{D}} = \sum_{k \leqslant \ell \leqslant k+n-1} a_\ell \varphi^{\ell}: \, \, X \to \cc
$$
extending $f$. If $f$ and $\varphi$ take values in a number field $F$ then $a_\ell \in F$ for all $\ell$. 
 
\section{Low dimensional examples}

In this section we discuss  low dimensional examples, which are relevant for the considerations in  sections 4 and 5.

\medskip

\subsection{The sphere $S^2$} \label{sphere}
Let 
\[
S^2 = \{ x \in \R^3: x_1^2 + x_1^2 + x_3^2= 1\}
\]
be the unit sphere in $\R^3$ and consider a set $\mathscr{D} = \{ 
x^{(i)} \in S^2: 1\leqslant i \leqslant n \}$
and a function \[f: \mathscr{D} \rightarrow \mathbb{C}. \] We will write
$y^{(i)} = f (x^{(i)})$ for $x^{(i)} \in \mathscr{D}$
and assume that the function $f$ is not constant. 
Thinking of this as an observable set of data with an inherent spherical symmetry it is natural to assume 
\[x^{(i)}\in \Q^3, \quad y^{(i)} \in \Q\] 
for all $1\leqslant i \leqslant n$, and we will impose this assumption throughout this section. 

\medskip \noindent
\underline{Our Problem:} \ \  Find an explicit function $W_{f,\mathscr{D}}$ interpolating the function $f$.

\medskip \noindent
We solve this using the representation theory of  the symmetry group $\operatorname{SO}(3)$ of the sphere.
We endow the sphere with the surface area measure $\mu,$ which is invariant under the action of $\operatorname{SO}(3),$  normalized such that $\mu(S^2)=1$ and consider the Hilbert space 
$L^2(S^2, \mu)$ of square integrable complex valued functions with respect to $\mu.$ 
The group $G=\operatorname{SO}(3)$ acts on $L^2(S^2)$ by 
$$
(g \cdot \psi)(x) = \psi (g^{-1}x), \quad g \in \operatorname{SO}(3), \psi \in L^2(S^2), \quad x \in S^2. 
$$
Let $\Delta =\sum_{i=1}^3 \partial^2/\partial x_i^2$ be the Laplacian on $\R^3.$ Recall that a harmonic polynomial on $\R^3$ 
is a polynomial $p = p(x_1,x_2,x_3) \in \mathbb{C}[x_1, x_2, x_3]$ in three variables
with complex coefficients satisfying 
    $$
    \Delta p = 0.
    $$
Example: Let $\ell \in \Z_{\geqslant 0}.$ The polynomial  
    $$
p(x_1,x_2,x_3 ) = (x_1 + \sqrt{-1}x_2)^\ell
    $$
is  harmonic of degree $\ell.$ 
Note that $p$ is independent of $x_3$ and is holomorphic as a function of the complex variable
$z=x_1 + \sqrt{-1}x_2.$ Therefore $p$ is a harmonic function of $(x_1,x_2)$ in the sense that 
$$
\left( \frac{\partial^2}{\partial x_1^2} + \frac{\partial^2}{\partial x_2^2} \right) p = 0. 
$$
In particular, 
$\Delta p = 0$.
The space of spherical harmonics of degree $\ell$ ($\ell \in \Z_{\geqslant 0}$) is the subspace
$V_\ell$ of $L^2(S^2)$ consisting of restrictions to $S^2$ of harmonic polynomials on $\R^3$ that are homogeneous of degree $\ell.$ 
Recall that the space $V_\ell$  of spherical harmonics of degree $\ell$ is an irreducible representation of $\operatorname{SO}(3)$ of degree $2\ell +1.$
If $V$ is any irreducible finite-dimensional representation of $\operatorname{SO}(3),$ then there is a unique integer $\ell$ such that $V$ is isomorphic to $V_\ell.$ The space $L^2(S^2)$ decomposes as a Hilbert space direct sum 
    $$
L^2(S^2) \simeq \hat{\bigoplus}_{\ell \in \Z_{\geqslant 0}} V_\ell. 
    $$
In the physics literature \cite[Definition 17.5]{Hall-Quantum}, $\ell$ is called the spin of the representation $V.$ In Lie theoretic language, the irreducible representations of $\operatorname{SO}(3)$ are classified by their highest weight \cite{Hall-Lie}, and $\ell$ is the highest weight of $V.$

   

 

\begin{theorem}
We keep the notation and assumptions. 
Suppose $(0,0,1) \notin \mathscr{D}$ and $(x_1^{(j)},x_2^{(j)}) \neq (x_1^{(i)},x_2^{(i)})$ whenever $i \neq j$ for all 
$x^{(i)}, x^{(j)} \in \mathscr{D}$. Then,
there are unique coefficients $a_1,\ldots a_\ell \in \Q(\sqrt{-1})$
such that the harmonic polynomial 
$$
W_{f,\mathscr{D}}(x_1,x_2,x_3) = \sum_{\ell = 1}^n a_\ell (x_1 + \sqrt{-1} x_2)^\ell
$$
satisfies
\[W_{f,\mathscr{D}}(x) = f(x)\] for all $x\in \mathscr{D}$.
\end{theorem}
\begin{proof}
This follows using the argument in  Section \ref{vandermondetrick}. 
\end{proof}

\medskip
\begin{remark}
We observe that the Whitney function $W_{f,\mathscr{D}}$ generates a finite-dimensional subrepresentation
of $L^2(S^3)$ which may not be irreducible.
\end{remark}

\medskip

\subsection{The upper half plane 
} \label{hyperbolic}
The upper half plane 
$$
\mathbb{H} = \{x+iy\in \cc: y>0\}
$$
is a bounded symmetric domain isomorphic to the Riemannian symmetric space $\operatorname{SL}(2,\R)/\operatorname{SO}(2)$, where $K=\operatorname{SO}(2)$ is the maximal compact subgroup of $\operatorname{SL}(2,\R)$. 

We fix again a finite 
set 
\[\mathscr{D}= \{  z^{(i)}\in \mathbb{H}: \ \ 1\leqslant i \leqslant n \} \]
and a nonconstant function $f: \mathscr{D} \rightarrow \mathbb{C}$. \newline

\noindent
\underline{Our Problem:} \ \  Find an explicit function $W_{f,\mathscr{D}}$ interpolating the function $f$.

\

\noindent
We will use 2 different methods: in the first method we use again finite dimensional representations:

The \emph{finite-dimensional representations} of $\operatorname{SL}(2,\R)$ have the following explicit model. 
Let $\cc[z]$ be the complex vector space of polynomials in one variable with complex coefficients. 
For any integer $n\geqslant 1$, let $\cc[z]_n$ be the subspace of polynomials $p(z)\in \cc[z]$ of degree $\ell < n$. 
The group $\operatorname{SL}(2,\R)$ acts on $\cc[z]_n$ as 
$$
\pi_n \begin{pmatrix}
    a & b \\ c & d 
\end{pmatrix} .p(z) = (-bz+d)^{n-1} p\left( \frac{az-c}{-bz+d} \right), \quad p(z)\in \cc[z]_n. 
$$
Consequently, $(\pi_n, \cc[z]_n)$ is a finite-dimensional representation of dimension $n$ over $\cc$. 
A basis for $\cc[z]_n$ consists of the monomials $1,z,\ldots, z^{n-1}$ and the action on a monomial of degree $\ell \leqslant n-1$ is 
$$
\pi_n \begin{pmatrix}
    a & b \\ c & d 
\end{pmatrix} z^\ell = (-bz+d)^{n-1-\ell} (az-c)^{\ell} \in \cc[z]_n. 
$$
The action is irreducible. Note that the action $\pi_n$ depends on $n$ and, for instance, 
$$
\pi_n \begin{pmatrix}
    a & -1 \\ 1 & 0 
\end{pmatrix} .1 = z^{n-1} \notin \cc[z]_{n-1}=\operatorname{span}_\cc \{1,z,\ldots,z^{n-2}\}.
$$
For a function $f: \mathscr{D}\to \cc$ 
the model $(\pi_n,\cc[z]_n)$ of a finite-dimensional representation of $\operatorname{SL}(2,\R)$ of degree $n$ can be used to construct a holomorphic Whitney extension \[W_{f, \mathscr{D}}: \mathbb{H} \to \cc\] 
so that 
$ W_{f, \mathscr{D}}(z) = f(z)$ for all $z\in \mathscr{D}$.

\

\noindent
\underline{Observation:}
In addition to $\mathbb{H}$, the group $\operatorname{SL}(2,\R)$ acts on the Shilov boundary
$$
\overline{\mathbb{H}} \cap (\operatorname{Im}(z)=0) \simeq \R 
$$
of $\mathbb{H}$ which is isomorphic to the real line. Consequently, 
if the function $f$ is defined on $n$ real points $x^{(i)} \in \R$, 
we may choose a model $(\pi_n,\cc[x]_{n})$ of the finite-dimensional representation of $\operatorname{SL}(2,\R)$ in which we regard $\cc[x]$ as polynomial functions of a real variable $x\in \R$. 
In this case, the Whitney extension theorem for $f$  
reduces to the classical fact that $n$ pairs of distinct points $(x_i,y_i) \in \R^2$ has a unique polynomial of degree $n-1$ through them.
Thus, our methods recover the polynomial interpolation theorems by Newton, Lagrange, etc. as special cases. 

\medskip

\noindent
In the second method we use \emph{infinite dimensional holomorphic representations}.
\medskip

Let $\mathcal{O}(\mathbb{H})$ be the space of holomorphic functions $f:\mathbb{H}\to \cc$ on the upper half plane. 
For integers $\ell \geqslant 2$, $\operatorname{SL}(2,\R)$ admits irreducible unitary infinite-dimensional discrete series representations
$(\pi_\ell^\pm,\mathscr{H}_\ell^\pm)$, where
$$ 
\mathscr{H}_\ell^+ = \{ f\in\mathcal{O}(\mathbb{H}): 
|| f||^2_\ell = \int_\mathbb{H}|f(x+iy)|^2 y^{\ell-2} dxdy < \infty \}
$$
is a Hilbert space with the action
$$
\pi_\ell^+  \begin{pmatrix}
    a & b \\ c & d 
\end{pmatrix}  f(z) = (-bz+d)^{-\ell} f\left( \frac{az-c}{-bz+d}\right).
$$
For each $\ell \geqslant 2$, the function $(z+i)^{-\ell} \in \mathscr{H}_\ell^+$ is an eigenfunction for the hyperbolic Laplacian. Moreover, $\{ (z+i)^{-\ell} \}_{\ell \geqslant 2}$ is a separating family of holomorphic functions on the upper half plane with the property that if $f$ is a function on \[
\mathscr{D} =\{ z^{(i)} \in \mathbb{H}: 1 \leqslant i \leqslant n\},\] then there is a holomorphic extension
$W_{f,{\mathscr{D}}}: \mathbb{H} \to \cc$
of the form 
$$
W_{f,\mathscr{D}}(z) = \sum_{\ell \geqslant 2} a_\ell (z+i)^{-\ell}, \quad z\in \mathbb{H}
$$
for some $a_\ell \in \cc$ such that 
$W_{f,\mathscr{D}}(z) = f(z)$ for all $z \in \mathscr{D}$.

\medskip

\begin{remark} 
\begin{enumerate} 
\item Both the sphere
$S^2$ and the upper half plane $\mathbb{H}$ can be seen as symmetric spaces arising from two different real forms of the same complex group $\operatorname{SO}(3,\cc)$. Indeed,
the sphere $S^2\simeq \operatorname{SO}(3)/\operatorname{SO}(2)$
is a compact symmetric space for the compact real form $\operatorname{SO}(3)$ of 
$\operatorname{SO}(3,\cc)$, 
while  $$\mathbb{H} \simeq \operatorname{SO}(2,1)^\circ/\operatorname{SO}(2) \simeq X(2,1)_{-}$$ 
is a noncompact symmetric space for the noncompact real form 
$\operatorname{SO}(2,1)$ of $\operatorname{SO}(3,\cc)$. 
Although the upper half plane fits into the family of real hyperboloids
$X(p,q)_{-}$, notice that the case $p=2$ and $q=1$ is excluded from Theorem \ref{whitney1}. This is because the untwisted space $L^2(X(2,1)_{-})$ does not contain discrete spectra. 
Instead, each of the functions $(z+i)^{-\ell}$ can be regarded as a minimal $K$-type associated with distinct bundles as $\ell$ varies. 
\end{enumerate}
\end{remark}

\medskip

\medskip

\subsection{The anti de Sitter space $X(2,2)_{+}$ } 
\label{AdS}
 We focus in this part on a rational data set $\mathscr{D}_n$  since in practical applications involving software, all data are rational numbers. First we prove an approximate  Whitney extension theorem, then we  show that we can in fact modify the functions to obtain an exact Whitney extensions. 

\subsubsection{The setup}
We consider  a set 
\[\mathscr{D}= \{x^{(i)}\in X(2,2)_+(\Q) : 1 \leqslant i \leqslant n \} \]

We say  that the  points in $\mathscr{D}$ 
are {\em in general position} if there is no orbit in $X(2,2)_+$ 
of a subgroup $G'$ isomorphic to one of the groups $\operatorname{SO}(1,2)^{\circ}$, $\operatorname{SO}(2,1)^\circ$,  $\operatorname{SO}(1,1)^\circ$, $\operatorname{SO}(2,0)$, or $\operatorname{SO}(0,2)$ which contains $\mathscr{D}$.  For example  the points whose first  coordinate is zero are not in general position.

Let \[
f: \mathscr{D} \rightarrow \mathbb{Q}\] 
be a nonconstant function and write $y^{(i)}=f(x^{(i)})$ for $x^{(i)} \in \mathscr{D}$.

\medskip
\subsubsection{Representation theory of $\operatorname{SO}(2,2)^\circ$} 
For the noncompact group $\operatorname{SO}(2,2)^\circ$ and the noncompact de Sitter space  the functions constructed by M. Flensted-Jensen \cite{Flensted-Jensen} will be used to construct  Whitney extensions.

\begin{proposition}
Let $X = \{ x \in \mathbb{R}^4: x_1^2 + x_2^2 -x_3^2 -x_4^2 = 1\}.$
For any integer $\ell \geqslant 2,$ define
$\psi_\ell^+: X \to \mathbb{C}$ and $\psi_\ell^-: X \to \mathbb{C}$ by
$$
\psi_\ell^+(x) = \left( x_1 + \sqrt{-1} x_2 \right)^{-\ell}, \quad 
\psi_\ell^-(x) = \left( x_1 - \sqrt{-1} x_2 \right)^{-\ell}.
$$
Then
$\psi_\ell^\pm$ are Flensted-Jensen functions for the parameter $\lambda=2\ell-2,$ and the $(\mathfrak{g}, K)$-modules generated by $\psi_\ell^+$ and $\psi_\ell^-$ form a complete set of discrete series representations for $X(2,2)_+.$

\end{proposition}

\begin{proof}
Cf. \cite{KKpoincare} Part 4, section 9, p. 215.    
\end{proof}

Suppose now that all the points in $\mathscr{D}$ are in general position.
 We would like to find coefficients $a_\ell^\pm$ such that 
\[
y^{(i)} = \sum_{\ell \geq 2} a_\ell^{\pm} \psi_\ell^{\pm} (x^{(i)})
\]
for all $x^{(i)} \in \mathscr{D}$ and where $a_\ell^{\pm} \neq 0$ for finitely many $\ell \geq 2.$ We have to choose the signs to be the same for all terms.  
Then 
$$
W_{f,\mathscr{D},\pm} = \sum_{\ell \geq 2} a_\ell^{\pm} \psi_\ell^{\pm}
$$
is a $K$-finite vector in the representation space $L^2(X(2,2)_+)$ with the property that 
$W_{f,\mathscr{D},\pm }(x) =f(x)$ for all $x \in \mathscr{D}$. 
Thus we construct two Whitney extensions, a holomorphic and an antiholomorphic one.

\medskip
\subsubsection{More about rational points of $X(2,2)_+$}
A matrix $A$ is called \emph{non-exceptional} if
\[
\det (I+A) \neq 0,
\] where $I$ is the identity matrix, and otherwise $A$ is called {\em exceptional}. 
Let $A$ be a non-exceptional matrix. Then the matrix $S$ defined by the equation
    $$
I + S = 2(I+A)^{-1} 
    $$
    is also non-exceptional, and
$A = (I-S)(I+S)^{-1}  = (I+S)^{-1}(I-S).$
Since the factors commute, we will follow Weyl \cite{Weyl} and use the rational notation 
$$
A = \frac{I-S}{I+S}.
$$
Note that $A\in \operatorname{SO}(n)$ if and only if 
$S$ is skew symmetric. 

\begin{proposition}
There are infinitely many rational points on the space 
 $$
X(2,2)_{+}  = \{x \in \mathbb{R}^4 :  x_1^2 + x_2^2 - x_3^2 - x_4^2 = 1\}. 
 $$
\end{proposition}
\begin{proof}
Any quadratic form $$f = a_1 x_1^2 + a_2 x_2^2 + a_3 x_3^2 + a_4 x_4^2 + a_5 x_5^2, \quad a_i \in \Q$$ 
over $\Q$ of rank 5 represents zero in $\Q$ if and only if it represents zero in $\mathbb{R}$ \cite[p. 43]{Serre}. 
In particular, there exist $x_i \in \Q$ such that
    $$
 0 =  x_1^2 + x_2^2 - x_3^2 - x_4^2  - x_5^2,
    $$
    $x_5 \neq 0, $ and 
 the point $x' = (x_1 / x_5, x_2 /x_5, x_3/ x_5, x_4/x_5) \in X(2,2)_{+} \cap \Q^4.$
For any proper orthogonal matrix $A\in \operatorname{SO}(2,\Q)$, the point
$$
\begin{pmatrix}
    A & \\
    & I
\end{pmatrix} 
x' \in X(2,2)_{+} \cap \Q^4.
$$ 
The statement follows. 
\end{proof}

\begin{lemma}
    Let $x = (x_1,x_2,x_3,x_4)$ and $x'=(x'_1,x'_2,x'_3,x'_4)$ be rational points in $X(2,2)_{+}$.
    Suppose 
    $$
    x_1 = x_1',\quad x_2 = x_2'.
    $$
    For any matrix $A \in \operatorname{SO}(2, \Q),$
    let 
    $$
\begin{pmatrix}
    x_1^\circ \\
    x_2^\circ
\end{pmatrix}
=
\begin{pmatrix}
   a & b \\
   c & d
\end{pmatrix}
\begin{pmatrix}
    x'_1 \\
    x'_2
\end{pmatrix}
=
\begin{pmatrix}
    ax'_1 + bx'_2 \\
    cx'_1 + dx'_2
\end{pmatrix}
    $$
    and let $x^\circ_3 =x'_3$ and $x^\circ_4 = x'_4.$
    Then the perturbed point $x^\circ = (x^\circ_1,x^\circ_2,x'_3,x'_4) $ is a rational point on $X(2,2)_{+}$.
\end{lemma}

\begin{proof}
    By assumption, 
    $$
(x'_1)^2 + (x'_2)^2 = 1 + (x'_3)^2 + (x'_4)^2,
    $$
    so $(x'_1,x'_2)$ lies on a circle of radius $1 + (x'_3)^2 + (x'_4)^2,$ and $\operatorname{SO}(2)$
    acts on this circle. 
\end{proof}

\begin{lemma} 
Let $\varepsilon \in \Q, \varepsilon > 0$, and define the skew-symmetric matrix
    $$
S = \begin{pmatrix}
    0 & \varepsilon \\
     - \varepsilon & 0 \\
\end{pmatrix}. 
    $$
    Then  
    $\det(I + S) = 1 + \varepsilon^2 > 0, $ so $S$ is non-exceptional, 
and 
    $$
A = \frac{I-S}{I+S} = 
\begin{pmatrix}
          1 & \varepsilon \\
     -\varepsilon & 1 \\
\end{pmatrix}^{-1}
\begin{pmatrix}
       1 & -\varepsilon \\
     \varepsilon & 1 \\
\end{pmatrix}
=\frac{1}{1+\varepsilon^2}
\begin{pmatrix}
          1 - \varepsilon^2  & -2\varepsilon \\
     2 \varepsilon & 1 - \varepsilon^2 \\
\end{pmatrix}  
    $$
    is a non-exceptional matrix in $\operatorname{SO}(2,\Q)$. 
Also, the distance between $A$ and the identity transformation $I$ is
$$
||A-I||^2 
= \frac{8\varepsilon^2}{1+\varepsilon^2},
$$
which tends to 0 as $\varepsilon$ tends to zero. 
\end{lemma}

\begin{lemma} \label{perturb}
    Let $x$ and $\tilde{x}$ be rational points in $X(2,2)_{+}$ whose first and second coordinates coincide: $x_1 = \tilde{x}_1$ and $x_2 = \tilde{x}_2.$
    Then the point 
    $$
\tilde{x}^{(\varepsilon)} =
\begin{pmatrix}
   \frac{1-\varepsilon^2}{1+\varepsilon^2} x_1 - \frac{2\varepsilon}{1+\varepsilon^2}x_2 \\
   \frac{2\varepsilon}{1+\varepsilon^2}x_1 + \frac{1-\varepsilon^2}{1+\varepsilon^2} x_2\\
    \tilde{x}_3\\
    \tilde{x}_4
\end{pmatrix}
    $$
    is a rational point on $X(2,2)_{+}$. Clearly, the points $\tilde{x}$
    and $\tilde{x}^{(\varepsilon)}$ are equidistant from $x,$ i.e.
    $$
||x - \tilde{x} ||_{\mathbb{R}^4} = ||x - \tilde{x}^{(\varepsilon)} ||_{\mathbb{R}^4}, 
$$ while  the distance between $\tilde{x}$ and the perturbed point $\tilde{x}^{(\varepsilon)}$
is given by   
$$
||\tilde{x} - \tilde{x}^{(\varepsilon)} ||^2_{\mathbb{R}^4} =
\frac{4\varepsilon^2}{(1+\varepsilon^2)^2}\left( (x_2-\varepsilon x_1)^2 +(x_1-\varepsilon x_2)^2 \right),
$$
which tends to zero as $\varepsilon$ tends to zero: Indeed, if $\varepsilon \leqslant 1,$ and 
if $C = \max_{i,j} |x_j^{(i)}| $
then 
$
||\tilde{x} - \tilde{x}^{(\varepsilon)} ||_{\mathbb{R}^4} \leqslant 2^{5/2} C \varepsilon.
$
Finally, 
$$x_1 \neq \tilde{x}_1^{(\varepsilon)}, \quad 
x_2 \neq \tilde{x}_1^{(\varepsilon)}.$$
\end{lemma}

\medskip
\subsubsection{An algorithm for Whitney functions.} \label{algorithm}

\begin{enumerate}
\item Define a small error tolerance $\varepsilon \in \Q, \varepsilon > 0$. 
    \item Loop through indices $i$ with $2\leqslant i \leqslant n$. 
    Let such an $i$ be fixed. We argue by induction  where the inductive hypothesis is that for all $j,k < i$
    ($j\neq k$) that 
    $$(x_1^{(j)},x_2^{(j)}) \neq (x_1^{(k)},x_2^{(k)}).$$
    (Note that the statement for $i=2$ which starts the induction is obvious).
    Suppose there is some $k < i$ such that 
      $$(x_1^{(k)},x_2^{(k)}) = (x_1^{(i)},x_2^{(i)}).$$
    In particular, $k< i$ is the only index with this property.  
    Choose $\varepsilon_i > 0$
    such that $\varepsilon_i < \varepsilon$
    and such that when acting  with the matrix  
     $$\kappa_i =\begin{pmatrix}
       A_{\varepsilon_i} & \\ & I_{2\times 2}
    \end{pmatrix} \in K $$
    on $x^{(i)}$ 
    (where $A_{\varepsilon_i}=(I-S_{\varepsilon_i})/(I+S_{\varepsilon_i})$ cf. Lemma \ref{perturb}), the point 
$x^{(i,\varepsilon_i)} := \kappa_i.x^{(i)} $  has the first two coordinates 
    $(x_1^{(i,\varepsilon_i)},x_2^{(i,\varepsilon_i)})$ \emph{distinct from all the points in  the finite list}
     $$
(x_1^{(j)},x_2^{(j)}),\quad j < i,
     $$
and such that $||x^{(i,\varepsilon_i)} - x^{(i)} ||_{\R^4} \leqslant \varepsilon_i < \varepsilon$.
    
We replace the point $x^{(i)}$ with the point $x^{(i, \varepsilon_i)}$.  
After having checked all points and replaced some with small perturbations if necessary, we obtain a new set $\mathscr{D}_{\varepsilon} = \{ x^{(i,\varepsilon_i)} \in X(2,2)_{+}(\Q): i\leqslant n \}$ of rational points in general position. We also define the function
\[f_{\varepsilon}:\mathscr{D}_{\varepsilon} \rightarrow \mathbb{Q}\]
by 
\[ f_{\varepsilon}(x^{(i, \varepsilon_i)}) =f(x^{(i)}) = y^{(i)}.\]
    \item 
     Section \ref{vandermondetrick}  implies that if we compute the column vector
    $$
    (a_\ell^+)_{\ell=2}^{n+1}
    =
  \left( (x_1^{(i, \varepsilon_i)} + \sqrt{-1}x_2^{(i,\varepsilon_i)})^{-\ell}  \right)_{1\leqslant i \leqslant n;\ 2 \leqslant \ell \leqslant n+1}^{-1} (y^{(i)})_{i=1}^n,
    $$
    where $a^+_\ell \in \Q(\sqrt{-1})$ for all $\ell$ then
    \[
    W_{f_\varepsilon,\mathscr{D}_{\varepsilon}} = \sum_{\ell=2}^{n+1} 
    a^+_\ell \psi_\ell^+
    \] 
    is a linear combination of $n$ linearly independent Flensted-Jensen functions with the property that 
   $W_{f_\varepsilon, \mathscr{D}_{\varepsilon}}(x^{(i,\varepsilon_i)}) =y^{(i)}$ for all $1\leqslant i \leqslant n$. 
\end{enumerate}
\medskip 

We will use the notation $W^\varepsilon_{f,\mathscr{D}} = W_{f_\varepsilon, \mathscr{D}_\varepsilon}$ and refer to $W^\varepsilon_{f,\mathscr{D}}$ as an {$\varepsilon$}-Whitney function of the data. 

\begin{theorem}[$\varepsilon$-Whitney extensions on $X(2,2)_+$] \label{X(2,2)}
   Let $\mathscr{D} = \{ x^{(i)} \in X(2,2)_+: 1 \leqslant i \leqslant n \}$ be in general position and let
   $f: \mathscr{D} \to \cc$ be nonconstant. 
   
Let $\varepsilon >0$.
There exists $\varepsilon_i \in \R$ 
such that $0 < \varepsilon_i <\varepsilon$ for all $1\leqslant i\leqslant n$
and a set of distinct points 
\[
\mathscr{D}_\varepsilon 
= \{x^{(i, \varepsilon_i)} \in X(2,2)_+ : 1\leqslant i \leqslant n \}\]
such that 
\[
\sup_i ||x^{(i, \varepsilon_i)}- x^{(i)} ||_{\R^4} < \varepsilon
\]
and such that the function 
$\varphi^+_1(x) := (x_1+\sqrt{-1}x_2)^{-1}$
is injective when restricted to $\mathscr{D}_\varepsilon$.

Define $f_\varepsilon: \mathscr{D}_\varepsilon \to \cc$
by $f_\varepsilon( x^{(i,\varepsilon_i)}) = f(x^{(i)})$.  
There exist $a_\ell^+ \in \cc$ 
and an 
analytic, $K$-finite, square integrable 
$\varepsilon$-Whitney function 
$W^\varepsilon_{f,\mathscr{D}} : X(2,2)_+ \to \cc$
of the form 
\[
W^\varepsilon_{f,\mathscr{D}} (x)
= \sum_{2\leqslant \ell \leqslant n+1} a_\ell^+ (x_1+\sqrt{-1}x_2)^{-\ell}, \quad
    x\in X(2,2)_+,
\]
with the property that 
$W^\varepsilon_{f,\mathscr{D}}(x)= 
f(x)$ for all $x \in \mathscr{D}_\varepsilon$.
\end{theorem}
\begin{remark}
    In Theorem \ref{X(2,2)} above, we choose a positive sign on the Flensted-Jensen functions and obtain a holomorphic extension and a Whitney representation in the holomorphic discrete series. The analogous theorem using antiholomorphic Flensted-Jensen functions $\varphi_\ell^{-}$ for all $\ell \geqslant 2$ is true as well.
\end{remark}

The coefficients on the Flensted-Jensen functions in a Whitney extension 
can be computed explicitly
in terms of the initial data by the formula in the corollary below. 
For variables 
$
\{ \alpha_1,\ldots,\alpha_k\}$, 
define
the elementary symmetric functions 
$$
e_m(\{ \alpha_1,\ldots,\alpha_k\})
= \sum_{1 \leqslant j_1<\cdots<j_m \leqslant k} \alpha_{j_1}\cdots \alpha_{j_m}, \quad m=0,1,\ldots, k.
$$ 
\begin{corollary} In the notation of Theorem \ref{X(2,2)} 
with $\mathscr{D}= \{x^{(i)}\in X(2,2)_+: 1\leqslant i \leqslant n\}$ and $f: \mathscr{D} \to \cc$, 
if the coordinates $(x_1,x_2)$ are distinct for all distinct $x\in \mathscr{D}$ and if $W_{f,\mathscr{D}}=\sum_\ell a_\ell^+ \varphi_\ell^+$,
then the coefficient $a_\ell^+$ is given by the formula
\begin{IEEEeqnarray*}{rcl}
a_\ell^+ &=& (x_1^{(\ell-1)}+\sqrt{-1}x_2^{(\ell-1)})^{2} \times \\
&\,& \sum_{j=1}^n f(x^{(j)}) 
\frac{(-1)^{n-(\ell-1)} e_{n-(\ell-1)} \left( \{ (x^{(i)}_1+\sqrt{-1}x^{(i)}_2)^{-1} : i\neq j \} \right) }{  \prod_{i: \ i\neq j}\left( (x^{(j)}_1+\sqrt{-1}x^{(j)}_2)^{-1}-(x^{(i)}_1+\sqrt{-1}x^{(i)}_2)^{-1} \right)}
\end{IEEEeqnarray*}
for all $2\leqslant \ell \leqslant n+1$.
\end{corollary}

\medskip
\subsubsection{From Whitney functions to Whitney representations}
  
Let $G = \operatorname{SO}(2,2)^{\circ}$ and $ H = \operatorname{SO}(1,2)^{\circ}$
and let $W_{f, \mathscr{D}}$ be a Whitney function on $G/H$.
Let $V$ be the $(\mathfrak{g}, K)$-module in $L^2(G/H)$ generated by $W_{f,\mathscr{D}}$. 
This is the vector space of $K$-finite functions in $L^2(G/H)$
obtained by acting on $W_{f,\mathscr{D}}$ 
with the universal enveloping algebra $U(\mathfrak{g}_\mathbb{C}).$
Let $\Pi_{n}$ be the representation of $G$ on the completion of $V$ in $L^2(G/H)$ and let $(\pi_\ell, V_\ell)$ ($\ell\geqslant 2$) be the Flensted-Jensen representation of $G$ in $L^2(G/H)$ generated by $\varphi_\ell^+.$
The results of M. Flensted-Jensen imply in the special case of $X(2,2)_{+}$ that if  $\ell, \ell' \geqslant 2$ are distinct, then 
the  representations $\pi_{\ell}$ and $ \pi_{\ell'}$  are not equivalent. 
For the Whitney function $W_{f,\mathscr{D}} = \sum_{\ell \leqslant n+1} a_\ell \varphi^+_\ell,$
$$
\Pi_n = \bigoplus_{\ell \leqslant n+1: a_\ell \neq 0} \pi_\ell.
$$
Note that if $a_\ell =0$ then the representation $\pi_\ell$ is not a direct summand.

\begin{lemma}
Define a hyperplane $H_k$ in $\mathbb{C}^n$ as
$$H_k =\{ z\in \mathbb{C}^n: z_k = 0\}, \quad 1\leqslant k\leqslant n.$$
The union $\cup_{k \leqslant n}H_k$ is a subset of $\mathbb{C}^n$ of Lebesgue measure zero. 
Let $\mathscr{D} = \{x^{(i)} \in G/H: 1\leqslant i \leqslant n\}$ 
be in general position and let 
$f: \mathscr{D} \to \cc$ be a nonconstant function. 
Fix $\varepsilon >0$ and choose modified points $\tilde{x}^{(i)} =x^{(i,\varepsilon_i)}=k_{\varepsilon_i}. x^{(i)}$ such that the matrix 
$$
T = \left( \varphi_\ell^+(\tilde{x}^{(i)})\right)_{i,\ell} \in \operatorname{GL}(n,\mathbb{C}). 
$$
Let $T^{-1}$ be the inverse of $T$ in $\operatorname{GL}(n,\mathbb{C}).$
Let $y^{(i)} = f(x^{(i)})$ for all $i$ and assume
$f: \mathscr{D}\to \cc$ satisfies 
$$ (y^{(i)})_{i \leqslant n} \in \mathbb{C}^n \setminus T(\cup_{k\leqslant n} H_k).$$
Then the $\varepsilon$-Whitney function $W^\varepsilon_{f,\mathscr{D}}= \sum a_\ell \varphi_\ell^+$ has the property that $a_\ell \neq 0$ for all $\ell.$
In particular, the  representation $\Pi_n$ which it generates  contains a maximal number of summands:  
$$
\Pi_n = \bigoplus_{2 \leqslant \ell \leqslant n+1 } \pi_\ell.
$$
\end{lemma}
\begin{proof}
If the vector $(y^{(i)})_i \in \mathbb{C}^n$ lives in the set 
$\mathbb{C}^n \setminus T \cup_{k \leqslant n} H_k$
(which has full Lebesgue measure in $\mathbb{C}^n$) 
then
    $$
(a_\ell)_\ell = T^{-1} (y^{(i)})_i  \notin \cup_{k \leqslant n} H_k
    $$
    which means that  $a_\ell \neq 0$ for all $\ell.$
\end{proof}

\begin{proposition}
    Let $\mathscr{D}=\{ x^{(i)}\in G/H :1\leqslant i \leqslant n\}$ be a set of points in general position, and let 
    $f: \mathscr{D}\to \cc$ be a nonconstant function. 
    Suppose there is a Whitney function 
    \[W_{f,\mathscr{D}} = \sum_{2\leqslant \ell\leqslant  n+1} a_\ell \varphi_\ell^+
    \] such that $W_{f,\mathscr{D}}(x)=f(x)$ for all $x\in \mathscr{D}$.

  Fix $\varepsilon >0$ and let $y^{(i)} = f(x^{(i)})$.
    There exists coefficients $a_\ell^{(\varepsilon)} \in \cc$ and  a function  
    $$\mathscr{W}^{\varepsilon}_{f, \mathscr{D}} = \sum_{2\leqslant \ell \leqslant n+1} a_\ell^{(\varepsilon)} \varphi_\ell^+ $$
    such that 
    $$|\mathscr{W}^{\varepsilon}_{f, \mathscr{D}}(x)- f(x)| < \varepsilon \quad \forall x \in \mathscr{D} $$
    and such that 
    $a_\ell^{(\varepsilon)} \neq 0$ for all $\ell \leqslant n+1.$ In particular
     the representation $\Pi_n^{(\varepsilon)}$ generated by $\mathscr{W}^{\varepsilon}_{f,\mathscr{D}}$ is of the form 
    $$\Pi_n^{(\varepsilon)} = \bigoplus_{2\leqslant \ell \leqslant n+1} \pi_\ell.$$
 \end{proposition}
\begin{proof}
The proof relies on the observation that set of regular points
$$
(\mathbb{C}^n)^{\operatorname{reg}} := \{ z \in \mathbb{C}^n: z_i \neq 0 \ \forall i \leqslant n\}
$$
is dense in $\mathbb{C}^n.$
First, we verify that the function 
$$
W: z = (z^{(i)})_i \mapsto (a_\ell(z)_\ell) \in \mathbb{C}^n, \quad (z^{(i)})_i \in (G/H)^n
$$
is an open mapping. The map $\varphi: (G/H)^n \to \mathbb{C}^{n \times n}$
$$
z= (z^{(i)})_i \mapsto \varphi(z) = ( (z_1^{(i)} + \sqrt{-1} z_2^{(i)})^{-\ell} )_{i,\ell}
$$
is an analytic, open mapping, and so is the restriction  $\varphi|_{\varphi^{-1}(\operatorname{GL}(n,\mathbb{C}))}$ to the open set $\varphi^{-1}(\operatorname{GL}(n,\mathbb{C}))$.
Likewise, the evaluation map $e_y: \operatorname{GL}(n,\mathbb{C}) \to \mathbb{C}^n$ defined by 
$e_y: T \mapsto T (y^{(i)})_i$ for $T\in \operatorname{GL}(n,\mathbb{C})$
and the inversion map $\iota: \operatorname{GL}(n,\mathbb{C}) \to \operatorname{GL}(n,\mathbb{C})$ are both open mappings.
By definition, 
$$
W = e_y \circ \iota \circ \varphi|_{\varphi^{-1}(\operatorname{GL}(n,\mathbb{C}))}
$$
so $W$ is also an open mapping from an open subset of $\mathbb{C}^n$ into $\mathbb{C}^n.$

Let 
$$S_\delta = \begin{pmatrix}
    0 & \delta \\
    -\delta & 0\\
\end{pmatrix}$$
be the non-exceptional skew symmetric matrix whose Cayley transform 
$A(S_\delta) = (I-S_\delta)/(I+S_\delta) \in \operatorname{SO}(2).$
Define 
$$
k_\delta =
\begin{pmatrix}
   \frac{I-S_\delta}{I+S_\delta} & \\
    & I
\end{pmatrix} \in K = \operatorname{SO}(2)\times\operatorname{SO}(2).
$$
Note that $x=(x^{(i)})_i \in \varphi^{-1}(\operatorname{GL}(n,\mathbb{C})),$ so there is some $m \in \Z_{\geqslant 1}$
such that $k_\delta.x \in \varphi^{-1}(\operatorname{GL}(n,\mathbb{C}))$ for all $\delta \in (-1/m,1/m).$ 
The open set 
$$U_{m,\varepsilon} = \left\{ W( k_\delta. x ): \delta \in (-1/m,1/m), \left| \sum_{2 \leqslant \ell\leqslant n+1} a_\ell(k_\delta . x)\varphi_\ell^+ (x^{(i)})-y^{(i)} \right| < \varepsilon\,  \forall i \right \}$$ 
is an open neighborhood of 
$(a_\ell(x))_\ell$ in $\mathbb{C}^n,$ and therefore it meets the dense set $(\mathbb{C}^n)^{\operatorname{reg}}.$  Choose some $\delta$ such that $(a_\ell(k_\delta.x))_\ell \in  U_{m,\varepsilon} \cap (\mathbb{C}^n)^{\operatorname{reg}}.$
Then the function 
$$
\mathscr{W}^{\varepsilon}_{f,\mathscr{D}} = \sum_{2 \leqslant \ell\leqslant n+1} a_\ell(k_\delta . x)\varphi_\ell^+
$$
has the desired properties. 
\end{proof}

\begin{remark}
    Note that given $\varepsilon>0$, 
    $\mathscr{W}^{\varepsilon}_{f,\mathscr{D}}$ does not generally coincide with the $\varepsilon$-Whitney function $ W^{\varepsilon}_{f, \mathscr{D}}$ introduced previously. 
\end{remark}

\medskip

We call the representation $\Pi_n^{(\varepsilon)}$ the {\em Whitney representation of the data.}   
Our discussion implies.
\medskip

\begin{theorem}
    We fix $\varepsilon > 0.$
If $\mathscr{D}_n \subseteq \mathscr{D}_{n+1}$ and if 
$f_{n+1}: \mathscr{D}_{n+1} \to \cc$ extends 
$f_{n}: \mathscr{D}_{n} \to \cc$, 
then  we can choose an
analytic, square integrable, $K$-finite function
$\mathscr{W}^{\varepsilon}_{f, \mathscr{D}, n+1} \in L^2(X(2,2)_{+})  $ such that 
\[
| \mathscr{W}^{\varepsilon}_{f, \mathscr{D}, n+1} (x^{(i)}) - y^{(i)} |< \varepsilon \]  
where $y^{(i)} = f_{n+1}(x^{(i)})$ for all $i \leqslant n+1$ 
and such that the representation  $\Pi_{n}^{(\varepsilon)}$ is a subrepresentation of $\Pi_{n+1}^{(\varepsilon)}.$
\end{theorem}

\medskip

\subsubsection{An exact Whitney theorem} \label{exactWhitney}
In  this subsection, we obtain an exact Whitney theorem. 
Let us keep the notation above where $\mathscr{D} = \{x^{(i)} \in X(2,2)^+ :1\leqslant i \leqslant n\}$ and $f: \mathscr{D} \to \cc$, 
and 
let 
$$
\kappa_t = \begin{pmatrix}
\cos(t) & \sin(t) \\
-\sin(t)& \cos(t) 
\end{pmatrix}.
    $$
By direct calculation, 
\begin{IEEEeqnarray*}{rCl}
\varphi_\ell^+(\kappa_t.x) &=& 
\left( x_1\cos(t)+x_2\sin(t) + \sqrt{-1}(x_1(-\sin(t)) +x_2\cos(t)) \right)^{-\ell} \\
&=& \left( x_1 (\cos(t)-\sqrt{-1}\sin(t)) + x_2
(\sin(t) + \sqrt{-1}\cos(t)) \right)^{-\ell} \\
&=& \left( x_1 (\cos(t)-\sqrt{-1}\sin(t)) + \sqrt{-1}x_2
( \cos(t) - \sqrt{-1}\sin(t)) \right)^{-\ell} \\
&=& e^{\sqrt{-1}(-t)(-\ell)} \left( x_1+\sqrt{-1}x_2 \right)^{-\ell} \\
&=& e^{\sqrt{-1}t\ell} \varphi_\ell^+(x).
\end{IEEEeqnarray*}

As a consequence of the algorithm \ref{algorithm}, given $\varepsilon>0$, there exist
real numbers $0 < \varepsilon_i < \varepsilon $ 
and an $\varepsilon$-Whitney function
$W^\varepsilon_{f,\mathscr{D}}$ such that 
$$
W^\varepsilon_{f,\mathscr{D}}(\kappa_{\varepsilon_i}. x^{(i)})= y^{(i)}
$$
for all $i\leqslant n$ where $y^{(i)} = f(x^{(i)})$. Then
\begin{IEEEeqnarray*}{rCl}
y^{(i)} &=& W^\varepsilon_{f,\mathscr{D}}(\kappa_{\varepsilon_i}.x^{(i)}) \\
&=& \sum_{\ell \leqslant n+1} a_\ell \varphi_\ell^+ ( \kappa_{\varepsilon_i}.x^{(i)}) \\
&=& \sum_{\ell \leqslant n+1} a_\ell e^{\sqrt{-1}\varepsilon_i\ell} \varphi_\ell^+(x^{(i)}) 
\end{IEEEeqnarray*}
We may choose a smooth function $t: X(2,2)_+ \to \R_{>0}$ 
such that $t(x^{(i)}) = \varepsilon_i$ for all $i\leqslant n$.
Define
\begin{align*}
\varphi_\ell^{\operatorname{new}}(x) &=
\left( e^{\sqrt{-1} \ell t} \otimes \varphi_\ell^+ \right) (x) \\
&= e^{\sqrt{-1} \ell t(x)} \varphi_\ell^+(x) \\
&= \left( e^{-\sqrt{-1} t(x)} (x_1+\sqrt{-1}x_2) \right)^{-\ell} \\
&= \varphi_\ell^+( \kappa_{t(x)}.x) 
\end{align*}
We note that if $\lambda$ is the left regular representation on $L^2(X(2,2)_+)$, we could write  
$\varphi_\ell^{\operatorname{new}} = \lambda(\kappa_{-t}) \varphi_\ell^+$; however, the parameter $t$ is an analytic function of $x\in X(2,2)_+$.

The conclusion is that the analytic function 
$$
W_{f,\mathscr{D}}(x) 
= \sum_\ell a_\ell \varphi_\ell^{\operatorname{new}} (x)
$$
satisfies 
$W_{f,\mathscr{D}}(x) = f(x)$ for all $x\in \mathscr{D}$. Note, however, that the Flensted-Jensen functions supporting $W_{f,\mathscr{D}}$ 
have been twisted from $\varphi_\ell^+$ to $\varphi_\ell^{\operatorname{new}}$.  
\bigskip


 \section{Finite-dimensional representations and Whitney extensions}
 In this section we generalize the examples  of the sphere \ref{sphere}  and of the hyperbolic space \ref{hyperbolic}.

\bigskip

\subsection{Spherical representations} 
\subsubsection{The setting}Let $G^{\mathbb{C}}$ be a classical complex reductive Lie group with  maximal compact subgroup $K_0$. In this section we will concentrate on the groups 
$$\operatorname{GL}(n,\mathbb{C}), \operatorname{SL}(n,\mathbb{C}), \operatorname{Sp}(n,\mathbb{C}),  \operatorname{SO}(2n,\mathbb{C}),  \operatorname{SO}(2n+1,\mathbb{C}).$$ The maximal compact subgroups of  $G^{\mathbb{C}}$ are 
$\operatorname{U}(n), \operatorname{SU}(n), \operatorname{Sp}(n), \operatorname{SO}(2n)$ and \linebreak $ \operatorname{SO}(2n+1)$, respectively. 
They are  the fixed points of the Cartan involution.  

Since the case of the symplectic groups  differs from the other cases and requires special attention, we will  omit its discussion in this section.  

\medskip

Let $T_0$ be a maximal torus of $K_0$. The centralizer of $T_0$ in $G^{\mathbb{C}}$ is a Cartan subgroup $H^{\mathbb{C}}$ of $G^{\mathbb{C}}$ with Lie algebra $\mathfrak{h}^\cc$.
The finite-dimensional holomorphic representations of $G^\mathbb{C}$ as well as the finite-dimensional representations of  $K_0$ are parametrized by highest weights which we consider as linear functionals on $\mathfrak{h}^{\mathbb{C}}.$ 
We use the conventions of Bourbaki to denote the roots, the simple roots, the dominant Weyl chamber and the dominant integral   weights.

We focus first on compact symmetric spaces $K_0/K'_0 $.  We are interested in the highest weights of the subrepresentations of $L^2(K_0/K'_0).$ 
The theorems of H. Schlichtkrull, E. van den Ban and M. Flensted-Jensen  do not exclude compact spaces, so we will make use of them. 

\medskip

{\em Assumption:} We will in this section 4. always assume that all  groups are connected.

\medskip

 Up to finite coverings the compact rank one symmetric spaces are

\medskip

\begin{itemize}

\item $\operatorname{U}(n)/\operatorname{U}(1)\operatorname{U}(n-1) \subset \operatorname{GL}(n,\mathbb{C})/\operatorname{GL}(1,\mathbb{C})\operatorname{GL}(n-1,\mathbb{C} )$
\item $\operatorname{SO}(2n)/\operatorname{SO}(1)\operatorname{SO}(2n-1)\subset \operatorname{SO}(2n,\mathbb{C})/\operatorname{SO}(1,\mathbb{C})\operatorname{SO}(2n-1,\mathbb{C}) $
\item $\operatorname{SO}(2n+1)/\operatorname{SO}(1) \operatorname{SO}(2n)\subset \operatorname{SO}(2n+1,\mathbb{C})/\operatorname{SO}(1,\mathbb{C})\operatorname{SO}(2n,\mathbb{C}) $
\item  $\operatorname{Sp}(n)/\operatorname{Sp}(1)\operatorname{Sp}(n-1) \subset \operatorname{Sp}(n,\mathbb{C})/ \operatorname{Sp}(1,\mathbb{C}) \operatorname{Sp} (n-1, \mathbb{C}). $ 
\end{itemize}

\medskip

\begin{remark}
Here we assume that $K'_0$ is the stabilizer in $K_0$ of the first coordinate.
\end{remark}

\medskip
\begin{remark} We limit our consideration in this article to the rank one symmetric spaces and will not consider all the spherical spaces discussed in the article ``Laplacians on spheres'' by Schlichtkrull, Trappa, and Vogan \cite{Schlichtkrull-Trapa-Vogan}
although we will use some of their ideas and formulas.
\end{remark}

We  consider the noncompact symmetric space $G'/H \cap G' $  of a real inner form $G'$ as an orbit  of $G'$ on the complex symmetric space $G^{\mathbb{C}}/K_0$; hence as a subset of a complex symmetric space.
We are again interested in the stabilizers of the first coordinate vector.

\medskip

\subsubsection{Finite-dimensional irreducible subrepresentations of $L^2(K_0/K'_0)$}

We generalize here the results about the sphere $S^2$ in \ref{sphere}. We will first determine the highest weights of the subrepresentations of $L^2(K_0/K'_0)$. 
We recall the results and formulas  in \cite{Schlichtkrull-Trapa-Vogan}.

The highest weights of subrepresentations of the square integrable functions on the compact symmetric space are as follows:
\begin{itemize}
\item If $K_0/K'_0 = \operatorname{U}(n)/\operatorname{U}(1)\operatorname{U}(n-1)$, $n>2$ the highest weight is
  \[(a,0,...0,-a), \ a \in \mathbb{Z}^+ \]  
  (here we use that $\operatorname{U}(1)$ commutes with $\operatorname{U}(n-1)$)
  
\item If $K_0/K'_0 =\operatorname{SO}(2n)/\operatorname{SO}(2n-1)$ the highest weight is
  \[(b,0,0, ....  ), \  b \in \mathbb{Z}^+ \]
  
\item If $K_0/K'_0 =\operatorname{SO}(2n+1)/\operatorname{SO}(2n)$ the highest weight is
  \[(b,0,0 ,...  ), \ b \in \mathbb{Z}^+ \]
 
\end{itemize}
We denote these representations by $\Pi_a$ respectively $\Pi_b$

We note that
\medskip 

{\bf Remark:}

{\em
\begin{itemize}
\item Let $G'$ be a noncompact real inner form of $G^{\mathbb{C}}.$  The highest weights of the finite-dimensional representations in  $C^\infty(G'/G'\cap K_0)$ correspond to  highest weights of irreducible finite-dimensional representations in the spectrum of the compact symmetric space. We denote this finite-dimensional representation by $\pi_a$, respectively $\pi_b$.
\item The $n$-th symmetric power of the representation with highest weight $(1,0,\dots , 0)$ has highest weight $(n,0,\dots , 0)$,
hence it is  a spherical representation. Its highest weight vector is an $n$-th power of the highest weight vector of the smallest spherical representation and so the representation  is isomorphic to $\pi_n$.
\end{itemize}
}

Since the considerations for unitary groups are very similar we concentrate from now on symmetric spaces for orthogonal groups.
Their infinitesimal characters of the representations are
\[(b +n, 2n/2-1,\dots , 0  )\] if $K_0 = \operatorname{SO}(2n)$ and
$$(b +n+1/2, (2n+1)/2 -1, \dots ,1/2 )$$
if $K_0= \operatorname{SO}(2n+1)$.




\medskip
The highest weight vector of the realization of the finite-dimensional representation in  $C^\infty(G'/H')$ is a coefficient 
\[ \psi_b(g)= \langle  \pi_b(g)v_0,v_b \rangle \] where $v_0 $ is the $H'$-invariant unit vector and $v_b$ is the unit highest weight vector of $\pi_b$.
Furthermore,
\[\psi_b(g) = (\psi_1(g))^b.\]

 \bigskip

 \subsection{Whitney functions}  We prove a Whitney extension theorem using the same idea as for the sphere in  \ref{sphere}.  We assume that $G'$ is a noncompact connected real inner form of the linear complex  semisimple group $G^{\mathbb{C}}$ of type $B_n$ or $C_n$.
 
\begin{theorem} \label{finite-dim}
Consider a set of rational points $\mathscr{D} = \{x^{(i)}\in G'/H': 1 \leqslant i \leqslant n \}$ 
and a nonconstant function $f: \mathscr{D} \to \Q$.  
Suppose that $\psi_1(x^{(i)}) \neq 0$ 
and 
$\psi_1(x^{(i)}) \neq \psi_1(x^{(j)})$
for all $1\leqslant i<j \leqslant n$.
There are coefficients 
$a_1,\ldots, a_\ell$ which are 
algebraic over $\Q$
such that  
$$
W_{f,\mathscr{D}} (x) = \sum_{\ell = 1}^n a_\ell \psi_1^\ell (x)
$$
satisfies
$W_{f,\mathscr{D}}(x) = f(x)$ for all $x\in \mathscr{D}$.
\end{theorem}
\begin{proof}
This is proved using the Vandermonde trick in \ref{vandermondetrick}.
\end{proof}

Note that Theorem \ref{finite-dim} holds for any manifold $M$ and any smooth function $\psi_1$ on $M$ as long as the values of $\psi_1$ at the points in $\mathscr{D}$ are distinct,  nonzero, and rational. However, 
by formulating the theorem for homogeneous spaces, the Whitney functions it produces can be given a representation-theoretic interpretation.  

\medskip
\bigskip
 
\section{The orthogonal symmetric spaces  $X(p,q)_{\pm}$}
In this section we generalize the example of the de Sitter and anti de Sitter space. 
For a summary about the representations of the indefinite orthogonal groups which we consider in this section, \cite{Kobayashi-elliptic} is a good reference.

  %
%
%
%
%
%

We sketch the generalization of  the Whitney extension theorems to the higher dimensional symmetric spaces 
$$ X(p,q)_{-}=
\operatorname{SO}(p,q)^{\circ}/\operatorname{SO}(p,q-1)^{\circ}
$$
for $p\geqslant 1$ and $q\geqslant 2$. The spaces $X(p,q)_{-}$ are also called real hyperboloids.

\medskip
{\em The setting:}
First, we clarify the definitions of the semisimple symmetric spaces that we will study in the sequel. We refer to \cite[Section 2.5]{Kobayashi-elliptic} for further details.
Define 
$$
X(p,q)_+ = \{ (x,y) \in \R^{p+q}: |x|^2-|y|^2 = 1  \}
$$ and 
$$
X(p,q)_{-} = \{ (x,y) \in \R^{p+q}: |x|^2-|y|^2 = -1  \}.
$$
There is a diffeomorphism 
$$
X(p,q)_{-} \simeq X(q,p)_+
$$
given by $(x,y)\mapsto (y,x).$
The Lie group $G=\operatorname{SO}(p,q)^{\circ}$ acts transitively on $X(p,q)_+$ and on $X(p,q)_{-},$ and there are $G$-diffeomorphisms
$$
X(p,q)_+ \simeq \operatorname{SO}(p,q)^{\circ}/\operatorname{SO}(p-1,q)^{\circ},$$
and 
$$ X(p,q)_{-} \simeq \operatorname{SO}(p,q)^{\circ}/\operatorname{SO}(p,q-1)^{\circ}. 
$$

We parametrize the symmetric space 
$X(p,q)_{-}$
in terms of hyperbolic coordinates. 
The lemma below, which gives the $KJH$-decomposition of the symmetric space
$X(p,q)_{-},$ motivates this reparametrization. 

\begin{lemma} \label{HJK}
Let $p\geqslant 1, q\geqslant 2$ and let  
$G=\operatorname{SO}(p,q)^{\circ}.$
There is a decomposition 
$$G = KJH,$$
where 
$$
H = \begin{pmatrix}
    g & \\
    & 1 
\end{pmatrix},\quad g\in \operatorname{SO}(p,q-1)^{\circ}, \quad 
K = \begin{pmatrix}
   k_1 & \\ 
    & k_2
\end{pmatrix}, \quad k_1\in \operatorname{SO}(p), k_2\in \operatorname{SO}(q)
$$
and 
$$
J = \operatorname{exp} 
\begin{pmatrix}
    & t\\
    \\
t & 
\end{pmatrix}
=
\begin{pmatrix}
    \cosh(t) &  & \sinh(t) \\
    \\
    -\sinh(t) & & \cosh(t) 
\end{pmatrix}, \quad t\in \R.
$$
In particular, $G/H \simeq X(p,q)_{-}$ is a symmetric space of rank one. 
\end{lemma}
\begin{proof}
Let $G=\operatorname{SO}(p,q)^{\circ}$
and let $H$ be the image in $G$ of 
$\operatorname{SO}(p,q-1)^{\circ}$ under the embedding 
$$
g \mapsto 
\begin{pmatrix}
    g & \\
    & 1 
\end{pmatrix},\quad g\in \operatorname{SO}(p,q-1)^{\circ}.
$$
Let
$
JH/H \subseteq G/H
$
be a maximal, split, abelian subspace of $G/H.$
We may decompose $G$ as  
$
G = KJH,
$
where 
$$
K = \begin{pmatrix}
   k_1 & \\ 
    & k_2
\end{pmatrix}, \quad k_1\in \operatorname{SO}(p), k_2\in \operatorname{SO}(q)
$$
is the maximal compact subgroup of $G$. In particular, 
$
\operatorname{rank} G/H = \dim_\R J = 1.
$
\end{proof}

We introduce the coordinates 
$$
\Phi(r,s,t) = \left( r_1 \sinh(t), \cdots, r_{p}\sinh(t), s_{1}\cosh(t), \ldots,
s_q\cosh(t) \right)
$$
where $\sum_{j=1}^{p} r_j^2 = 1$ and 
$\sum_{j=1}^{q} s_j^2 = 1$. 

\medskip
Let $G = \operatorname{SO}(p,q)^\circ$ and $H = \operatorname{SO}(p,q-1)^\circ.$
Recall that the points in $\mathscr{D} = \{x^{(i)} \in G/H: 1 \leqslant i \leqslant n \}$  are \emph{in general position} if not all points are contained 
in an orbit $\tilde{G}/\tilde{G}\cap H$ in $G/H$ of a proper reductive subgroup $\tilde{G}$ of $G$ isomorphic to one of the orthogonal groups  
$\operatorname{SO}(p',q')^\circ$ for some $p' < p$ and $q' \leqslant q$.
\medskip

\emph{Flensted-Jensen functions}\newline 
If $p=2$ and $q=2$, the Flensted-Jensen functions $\varphi_\ell^+$ on the space
\begin{IEEEeqnarray*}{rCl}
X &=& X(2,2)_{+} = \{ x\in \R^4: x_1^2 + x_2^2 -x_3^2 - x_4^2 = 1 \} \\
&\simeq & X(2,2)_{-} = \{ (x_3,x_4,x_1,x_2) \in \R^4: x_3^2 + x_4^2  - x_1^2 - x_2^2  = -1\} 
\end{IEEEeqnarray*}
can be expressed in the new coordinates as follows:
\begin{IEEEeqnarray*}{rCl}
\varphi_\ell^+ (x) = (x_1 + \sqrt{-1} x_2)^{-\ell} &=& (x_1-\sqrt{-1}x_2)^{\ell} \frac{1}{(x_1^2 + x_2^2)^\ell} \\
&=& (s_1-\sqrt{-1}s_2)^{\ell}\cosh(t)^{\ell} \frac{1}{ ( s_1^2+s_2^2)^\ell \cosh(t)^{2\ell}} \\
&=& (s_1-\sqrt{-1}s_2)^{\ell} \cosh(t)^{-\ell}, \quad \ell \geqslant 2.
\end{IEEEeqnarray*}
We see that $\varphi_\ell^+$ is a product of a spherical harmonic polynomial 
$$p_\ell(s_1,s_2) = (s_1-\sqrt{-1}s_2)^{\ell}$$ of degree $\ell$
and a negative power of a hyperbolic cosine. 
This form of the Flensted-Jensen functions generalizes to symmetric spaces $$X(p,q)_{-} = \operatorname{SO}(p,q)^{\circ}/\operatorname{SO}(p,q-1)^{\circ}$$ of arbitrary signature $(p,q)$ where $p\geqslant 1$ and $q \geqslant 2$. 
The Flensted-Jensen functions on $X(p,q)_-$ take the form 
$$
\varphi_\ell (s,t)= p_\ell(s_1,\ldots, s_q) \cosh{(t)}^{-\ell-q}, \quad (s_1\ldots,s_q)\in S^{q-1} 
$$
for a harmonic polynomial $p_\ell$ of degree $\ell$ \cite[Section 8.4.2]{Schlichtkrull}.
For $\ell+1-(p-q)/2>0$, the functions $\varphi_\ell$ are square integrable. This is equivalent to $\ell+q \geqslant (p+q)/2$, $\ell \in \Z$.
Recall that a vector $c=(c_1,\ldots,c_q) \in \cc^q$ is said to be isotropic if 
$$
c_1^2+\ldots +c_q^2 =0.
$$
The space of spherical harmonics on $S^{q-1}$ of degree $\ell$ has a basis consisting of all 
polynomials
$$
(s_1c_1+\cdots +s_q c_q)^{\ell}
$$
for isotropic vectors $c \in \cc^q$ \cite{Helgason2}.

\medskip
\subsection{Whitney functions on $X(p,q)_{-} = \operatorname{SO}(p,q)^{\circ}/\operatorname{SO}(p,q-1)^{\circ}$}

\medskip

\begin{theorem}
Let $p \geqslant 1$, $q\geqslant 2$, and
let $K = \operatorname{SO}(p) \times \operatorname{SO}(q)$ be the maximal compact subgroup of 
$G = {\operatorname{SO}(p,q)}^{\circ}.$
Let $H = \operatorname{SO}(p,q-1)^{\circ}$,
let $\mathscr{D} = \{x^{(i)} \in X(p,q)_{-}: 1 \leqslant i  \leqslant n\}$ where 
$X(p,q)_{-} \simeq G/H$, and suppose 
$f:\mathscr{D} \to \cc$ is not constant. 
For $k =\lfloor q/2 \rfloor$ 
and $\ell \in \Z$, let 
\[
\varphi_\ell(s, t) =(s_1 + \sqrt{-1}s_2 +\cdots+ s_{2k-1}+\sqrt{-1}s_{2k})^{\ell} \cosh(t)^{-\ell-q}
\]
be Flensted-Jensen functions on $X(p,q)_{-}$.
If $\varphi_1(x^{(i)})$ are distinct as $i$ traverses $1,\ldots,n$, then 
there exists an analytic, $K$-finite, square integrable function $W_{f, \mathscr{D}}: X(p,q)_{-} \to \cc$ such that 
$W_{f,\mathscr{D}}(x) = f(x)$ for all $x \in \mathscr{D}$.
The function may be chosen as a $\mathbb{C}$-linear combination 
$$
W_{f,\mathscr{D}} = \sum_{ \ell  \geqslant (p-q)/2} a_\ell \varphi_\ell  
$$
of Flensted-Jensen functions.
\end{theorem}

\begin{proof}
Cf. Section \ref{vandermondetrick}. 
\end{proof}

\begin{remark}
When $p=2$, $\operatorname{SO}(p,2)^\circ/\operatorname{SO}(p,1)^\circ$ is an affine symmetric space of Hermitian type, and such spaces have been classified and studied further in  
\cite{Orsted-Olafsson1}, \cite{Orsted-Olafsson2}. Affine symmetric spaces of Hermitian type are also called compactly causal symmetric spaces. 
The Flensted-Jensen functions on these spaces admit closed formulas involving consecutive powers of hyperbolic cosine, indexed by weights \cite[Theorem 5.2]{Orsted-Olafsson2}, consistent with the formulas presented here.  
Consequently, explicit formulas for Whitney extensions  
can be constructed from functions in the holomorphic discrete series of other affine symmetric spaces of Hermitian type such as 
$\operatorname{SU}(p,q)/\operatorname{SO}(p,q)$.
\end{remark} 

\medskip
\subsection{Whitney representations} \label{Whitneyreps}
As in the classical cases of the sphere in section \ref{sphere} and of anti de Sitter space in section \ref{AdS} we associate a unitary representation  of the symmetry group $\operatorname{SO}(p,q)^{\circ}$ to data $\mathscr{D}$ in the symmetric space $\operatorname{SO}(p,q)^{\circ}/\operatorname{SO}(p,q-1)^\circ$ in general position on which a nonconstant function $f$ is defined. 

\medskip
Let $W_{f, \mathscr{D}}$ be the associated Whitney function and 
let $V$ be the $(\mathfrak{g}, K)$-module  generated by $W_{f,\mathscr{D}}$. It defines a unitary admissible representation which we again call the 
Whitney representation of the data $(f, \mathscr{D})$.  It is a direct sum of discrete series representations of $\operatorname{SO}(p,q)^{\circ}$. 
\medskip

\subsection{
Orbits of orthogonal subgroups}
There is a nested sequence of orthogonal symmetric spaces 
$$
X(2,q)_{-} \subset X(3,q)_{-} \subset \cdots  X(p-1,q)_{-} \subset X(p,q)_{-}.
$$
Let $p_0 \leqslant p$ be the smallest integer such that all the points in $\mathscr{D}$ belong to $X(p_0,q)_{-}$ and such that 
the data is in general position for the space $X(p_0,q)_{-}.$
We may then 
observe that the functions 
$
\varphi_\ell
$
for $$\ell=(p-q)/2 + \delta, (p-q)/2 + \delta + 1,\ldots, (p-q)/2+\delta+(n-1),$$
where $\delta=1/2$ if $p-q$ is odd and $\delta =0$ if $p-q$ is even,
form a list of $n$ independent Flensted-Jensen functions {\em for all the spaces} $X(p',q)_{-}$ where $p_0 \leqslant p'\leqslant p$.
In particular, there exists a Whitney function 
$$
W_{f, \mathscr{D}}= \sum_{(p-q)/2 +\delta \leqslant \ell \leqslant (p-q)/2 +\delta+ n-1} a_\ell \varphi_\ell
$$
on the ambient space $X(p,q)_{-}$ with the property that when restricted to $X(p_0,q)_-$, 
$W_{f, \mathscr{D}}$ can be regarded as a Whitney function for the space $X(p_0,q)_-$.

\medskip
\underline{Remark:} Suppose that the data are in an orbit of an orthogonal subgroup isomorphic to $\operatorname{SO}(p',q')^\circ$ for $p' < p$ and $q'< q$.
We conjecture that it is  best to proceed in two steps: 
 \begin{enumerate}
     \item In the first step, we use our algorithm to obtain a Whitney function for the proper orthogonal subspace $X(p',q')_{-}$ of $X(p,q)_{-}$. 
     \item In the second step, we use the results on holography by Kobayashi and Pevzner \cite{Pevzner} to obtain an extension to $X(p,q)_{-}$ and the  group $\operatorname{SO}(p,q)^\circ$. The restriction of the Whitney representations to orthogonal subgroups are known by the work of T. Kobayashi \cite{Kobayashi-elliptic} and we may use this to understand the extension of a spherical Whitney function from a subgroup to the larger group. 
 \end{enumerate}

\section{Whitney extensions on  some  reductive groups}

In this section, we prove explicit formulas for exact Whitney extensions
of functions defined on reductive groups $G$ admitting holomorphic discrete series representations. Note that we may always think of $G$ as a symmetric space: the smooth automorphism 
\begin{IEEEeqnarray*}{rCl}
G \times G &\longrightarrow& G\times G, \\
(g,h) &\mapsto& (h,g), \quad g,h \in G
\end{IEEEeqnarray*}
is an involution, and its fixed points is the diagonal subgroup: 
$$
(G\times G)^\sigma = \{ (g,g): g\in G\}.
$$
In particular, the smooth epimorphism
\begin{IEEEeqnarray*}{rCl}
G \times G &\longrightarrow& G, \\
(g,h) &\mapsto& g^{-1}h \quad g,h \in G
\end{IEEEeqnarray*}
descends to an isomorphism between the symmetric space and the group:
$$
(G\times G)/(G\times G)^\sigma \simeq G. 
$$

Further details on reductive groups with holomorphic discrete series representations relevant to this section can be found in \cite{Neeb} and \cite[Ch. VI]{Knapp2}. The results in this section can be seen as a generalization of the Whitney extension theorem for the example
$\operatorname{SO}(2,2)^\circ/\operatorname{SO}(1,2)^\circ$ in the following way: The 3-dimensional anti de Sitter space is a group manifold, since there are isomorphisms 
$$
\operatorname{SO}(2,2)^\circ/\operatorname{SO}(1,2)^\circ \simeq \operatorname{SL}(2,\R) \simeq \operatorname{SU}(1,1).
$$
The reductive group $\operatorname{SU}(1,1)$ has the property that its maximal compact subgroup $K$ contains an analytic subgroup $K_0 \simeq S^1$, where $S^1$ is the circle. Therefore, it is natural to consider the problem of extending the results above to more general reductive groups $G$ whose maximal compact subgroup contains a nontrivial center. 

Let $N\geqslant 1$ be an integer and let $\mathcal{E}$ be any open and bounded subset of $\cc^N$.
Let $\mathcal{O}(\mathcal{E})$ be the space of holomorphic functions on $\mathcal{E}$. Then
$$
\mathscr{H} = \{ f\in \mathcal{O}(\mathcal{E}): \int_\mathcal{E} |f(z)|^2 dz <\infty \}
$$
(where $dz$ is Lebesgue measure)
is a Hilbert a space with respect to the inner product
$$
\langle f, g \rangle = \int_\mathcal{E} f(z)\overline{g(z)} dz.
$$
Indeed, the limit of a convergent sequence remains holomorphic. 
Let $\{ \phi_k\}$ be an orthonomal basis for $\mathscr{H}$.
Define 
$$
K(z,w) = \sum_k \phi_k(z) \overline{\phi_k(w)},\quad z,w \in \mathcal{E}.
$$
Then $K$ is a reproducing kernel:  $K$ is holomorphic in $z$ and antiholomorphic in $w$, $K$ is positive definite, and for any $f\in \mathscr{H}$,
$$
f(w) = \langle f , K_w \rangle = \int_\mathcal{E} f(z) \overline{K(z,w)} dz 
$$
where $K_w(z) :=K(z,w)$. 
If $\Phi:\mathcal{E} \to \mathcal{E}$ is holomorphic and invertible with  
$\Phi^{-1}: \mathcal{E}\to \mathcal{E}$ holomorphic, $\Phi$ is said to be a biholomorphic automorphism.
Define the group
$$
G = \operatorname{Aut}(\mathcal{E}) := \{ \Phi: \mathcal{E}\to \mathcal{E} ; \textrm{$\Phi$ is biholomorphic}\} 
$$
with composition of functions as the group law. 
The Lie group $G$ acts on $\mathcal{E}$ as $\Phi. z = \Phi(z)$. 

\medskip
\underline{Example}: 
Consider the open bounded unit disk 
$$\mathcal{E} =\{ z\in \cc: |z| < 1 \}.$$ 
Then the reproducing kernel is
$$
K(z,w) = \frac{1}{(1-z\bar{w})^2}
$$
and $G =\operatorname{Aut}(\mathcal{E}) \simeq \operatorname{SU}(1,1)$.   
The group $\operatorname{SU}(1,1)$ acts on $\mathcal{E}$ by fractional linear transformations:
$$
\begin{pmatrix}
    \alpha & \beta \\
    \overline{\beta} & \overline{\alpha} 
\end{pmatrix}. z = \frac{\alpha z + \beta}{\overline{\beta}z+\overline{\alpha}},\quad z\in \mathcal{E}.
$$
This action is transitive, and the stabilizer of $z=0 \in \mathcal{E}$ is
the maximal compact subgroup
$$
K =\left\{ \begin{pmatrix}
    \alpha & \\
    & \bar{\alpha}\end{pmatrix} : |\alpha|^2 = 1 \right\}
 \simeq U(1).
$$
Hence $G/K \simeq \mathcal{E}$.

Returning to the general theory, we proceed by flipping the perspective and begin with a reductive Lie group $G$ with maximal compact subgroup $K$. 
Let $\mathfrak{g}$ and $\mathfrak{k}$ be the Lie algebras of $G$ and $K$, respectively. Suppose the Lie group $G$ is of \emph{Harish-Chandra type}: if $\mathfrak{c}$ is the center in $\mathfrak{k}$, and if 
$Z_\mathfrak{g}(\mathfrak{c})$ denotes the centralizer in $\mathfrak{g}$ of $\mathfrak{c}$, then 
$$
Z_\mathfrak{g}(\mathfrak{c}) = \mathfrak{k}.
$$  
Consider the Cartan decomposition $\mathfrak{g} = \mathfrak{k}\oplus \mathfrak{p}$, where $\mathfrak{k}$ is the $+1$ eigenspace and $\mathfrak{p}$ is the $-1$ eigenspace. 
The complexification of $\mathfrak{g}$ splits into a direct sum 
$$\mathfrak{g}^\cc = \mathfrak{k}^\cc \oplus \mathfrak{p}^+ \oplus \mathfrak{p}^-,$$
where $\mathfrak{p}^\pm$ are eigenspaces with eigenvalues $\pm i$, 
called the Harish-Chandra decomposition. 
Let $P^+$, $K^\cc$, and $P^-$ be the connected Lie subgroups of $G^\cc$ with Lie algebras $\mathfrak{p}^+$, $\mathfrak{k}^\cc$ and $\mathfrak{p}^-$, respectively; in particular, 
$P^+ K^\cc P^- \subset G_\cc$.
For the example $G=\operatorname{SU}(1,1)$, 
$$
P^+ = \left \{ \begin{pmatrix}
    0 & z \\
    0 & 0
\end{pmatrix} \right \},\quad P^- = \left\{ \begin{pmatrix}
    0 & 0 \\
    w & 0
\end{pmatrix} \right\}.
$$

The key feature of a group of Harish-Candra type is that 
$$
G \subseteq P^+ K^\cc P^-.
$$
The group 
$$
Q^\cc= K^\cc P^-
$$
is a parabolic subgroup of $G^\cc$ called the Harish-Chandra parabolic and the quotient $G^\cc/Q^\cc$ is a flag manifold. 
Harish-Chandra observed that  $G/K$ embeds in $G^\cc/Q^\cc \simeq P^+$, and after passing through the logarithm, $G/K$ can be identified with an open, connected subset $\mathcal{E}$ of the complex, finite-dimensional vector space $\mathfrak{p}^+$.
The set $\mathcal{E}$ is a \emph{bounded symmetric domain}, which is an open, connected subset of $\cc^N$ that is biholomorphic with the 
open unit ball 
$$B_1(\cc^N) = \{ z\in \cc^N: ||z|| <1 \}.$$

\medskip
{\em Holomorphic discrete series: }
Let $G$ be a reductive group with maximal compact subgroup $K$. Suppose $G$ is a group of Harish-Chandra type. Then $G/K = \mathcal{E}$ is a bounded symmetric domain.  
Define $\kappa: P^+ K^\cc P^{-} \to K^\cc$ to be the projection onto the $K^\cc$-factor. 
For $z\in \mathfrak{p}^+$ and $g\in G^\cc$ such that
$g \exp(z) \in P^+ K^\cc P^{-}$ define
$$
J(g,z) = \kappa(g \exp(z)).
$$
The function $J(g,z)$ is called the Jacobian factor, or the canonical automorphic factor, of $G^\cc$.
For $g\in G^\cc$, define $g^* = \overline{g^{-1}}$. 
For $z,w \in \mathfrak{p}^+$ 
such that $\exp(w)^* \exp(z) \in P^+ K^\cc P^-$, 
define the reproducing kernel 
$$
K(z,w) = J( \exp(w)^*, z)^{-1}. 
$$
Let $(\rho_\lambda, V)$ be a holomorphic representation 
of $K^\cc$ of highest weight $\lambda$, where $V$ is a finite-dimensional complex vector space. 
Let $J_\lambda = \rho_\lambda \circ J$. 
For $f$ in the space $\mathcal{O}(\mathcal{E}, V)$ of holomorphic functions on $\mathcal{E}$ valued in the complex vector space $V$, define 
$$
\pi_\lambda (g) f(z) = J_\lambda(g^{-1},z)^{-1} f(g^{-1}z), \quad
g \in G, z\in \mathcal{E}. 
$$
Let $\mu_\mathcal{E}$ be the invariant measure on $\mathcal{E}$, and let 
$K^\lambda =  \rho_\lambda \circ  K$. 
Note that the space of polynomials $\operatorname{Pol}(\mathfrak{p}^+, V)$ can be regarded as a subspace of $\mathcal{O}(\mathcal{E},V)$, since $\mathcal{E} \simeq G/K \subset P^+$. 
If $\langle \cdot, \cdot \rangle_V$ is an inner product on $V$, then  
$$
\langle f, g \rangle_\lambda = \int_\mathcal{E}  \langle  
K^\lambda(z,z)^{-1} f(z)
, g(z) \rangle_V d\mu_\mathcal{E}(z)
$$
defines an inner product on $\mathcal{O}(\mathcal{E},V)$, 
and 
$$
\mathscr{H}_\lambda = \{ f\in \mathcal{O}(\mathcal{E},V): \langle f,f\rangle_\lambda < \infty \}
$$
is a Hilbert space. Furthermore, $\pi_\lambda$ defines a unitary representation of $G$ on $\mathscr{H}_\lambda$. In fact, $(\pi_\lambda, \mathscr{H}_\lambda)$
is a realization of the unitary highest weight representation of $G$ with highest weight $\lambda$.  

The $K$-finite vectors of $\pi_\lambda$ is the space of polynomials $\operatorname{Pol}(\mathcal{E},V)$, which is a dense subspace of $\mathscr{H}_\lambda$.
In practice, an orthonormal basis of the space $\mathscr{H}_\lambda$ can be computed by applying Gram-Schmidt orthogonalization to a basis of $\operatorname{Pol}(\mathcal{E},V)$. 
In turn, this orthonormal basis can be used to define the reproducing kernel.

\subsection{Explicit formulas for holomorphic extensions on some reductive groups}

\subsubsection{ Example 1:  $G= \operatorname{SU}(p,q)$} 
When $G = \operatorname{SU}(p,q)$, then $G^\cc = \operatorname{SL}(p+q,\cc)$ and 
the maximal compact subgroup of $G$ is
$$
K = S(U(p)\times U(q))= \left\{ \begin{pmatrix}
    A &  \\
     & D 
\end{pmatrix} \in \operatorname{SL}(p+q,\cc): A^* = A^{-1}, D^*= D^{-1}
\right \}.
$$
The reductive group $\operatorname{SU}(p,q)$ is of Harish-Chandra type, since 
$K \simeq S^1 \times K_{ss}$, where $K_{ss}$ is semisimple, and the factor isomorphic to a circle $S^1$ in $K$ is 
$$
\begin{pmatrix}
    e^{i \varphi q} & \\
    & \ddots & \\
    & & e^{i\varphi q} \\
    & & & e^{-i\varphi p} \\
    &&&&\ddots \\
    &&&&& e^{-i\varphi p}
\end{pmatrix}.
$$
Furthermore, letting $I_p \in \operatorname{GL}(p,\R)$ and $I_q \in \operatorname{GL}(q,\R)$ denote identity matrices,
$$
P^+ = \left\{  
\begin{pmatrix}
    I_p & Z \\ 
    & I_q 
\end{pmatrix}: Z\in M(p,q,\cc)
\right\}, \quad 
P^- = \left\{  
\begin{pmatrix}
    I_p &  \\ 
    W & I_q 
\end{pmatrix}: W\in M(q,p,\cc)
\right\}.
$$
Moreover, $G = \operatorname{SU}(p,q)$ is contained in 
$$
P^+ K^\cc P^- = 
\left \{\begin{pmatrix}
    A &  B\\ 
     C & D
\end{pmatrix} \in \operatorname{SL}(p+q,\cc): \det D \neq 0 \right \}
$$
and the projection $\kappa$ onto the $K^\cc$-factor is 
$$
\kappa 
\begin{pmatrix}
    A & B \\ C & D
\end{pmatrix} = 
\begin{pmatrix}
    A-BD^{-1}C &  \\  & D
\end{pmatrix}.
$$
The symmetric domain $\mathcal{E}$ is 
$$
G/K \simeq \mathcal{E}:=\{ Z \in M(p,q,\cc) : ||Z|| < 1 \},
$$
and the canonical reproducing kernel is 
$$
K_\mathcal{E}(Z,W) = 
\begin{pmatrix}
    I_p-ZW^* & \\ 
    & (I_q-W^*Z)^{-1}
\end{pmatrix} \in K^\cc
$$
where  $Z,W \in \mathcal{E}$
and the canonical automorphic factor is 
$$
J(g, Z) = 
\begin{pmatrix}
    A - (AZ+D)(CZ+D)^{-1} C & \\ 
    & CZ+D
\end{pmatrix},\quad g = \begin{pmatrix}
    A & B \\ C & D
\end{pmatrix}.
$$
Consider the holomorphic characters of $K^\cc$ given by  
$$
\rho_\lambda
\begin{pmatrix}
    A &  \\ 
     & D
\end{pmatrix}
 = (\det D)^{\lambda} = (\det A)^{-\lambda},
 \quad
 \begin{pmatrix}
    A &  \\ 
     & D
\end{pmatrix} \in 
 K^\cc
$$
parametrized by $\lambda$; in the sequel, we will take $\lambda \in \Z$.
Given such a character $\rho_\lambda: K^\cc\to \cc^\times$, the reproducing kernel is 
$$
K^\lambda(Z,W) = \rho_\lambda \circ K_\mathcal{E}(Z,W) = \det(I_q-W^*Z)^{-\lambda}
$$
and the automorphic factor is 
$$
J_\lambda(g,Z) =\rho_\lambda \circ J(g,Z) = (\det (CZ+D))^{\lambda}.
$$
The invariant measure on $\mathcal{E}$ is 
$$
d \mu_\mathcal{E} =K^{p+q}(Z,Z) dZ = \det(I_q-Z^*Z)^{-(p+q)}dZ
$$
while $K^{\lambda}(Z,Z)^{-1} = \det(I_q-Z^*Z)^{\lambda}$ and the Hilbert space $\mathscr{H}_\lambda$ becomes  
$$
\mathscr{H}_\lambda = \{ f \in \mathcal{O}(\mathcal{E}, \cc); \int_\mathcal{E} |f(Z)|^2 \det(I_q-Z^*Z)^{\lambda-(p+q)} dZ < \infty \}.
$$
The 
holomorphic discrete series representation of highest weight $\lambda$ is
\begin{IEEEeqnarray*}{rCl}
\pi_\lambda(g)  f (Z)  &=&J_\lambda(g^{-1},Z)^{-1} f(g^{-1}.Z) \\ 
&=&
\det((CZ+D))^{-\lambda} f\left( (AZ+B)(CZ+D)^{-1} \right)
\end{IEEEeqnarray*}
for $g^{-1} = \begin{pmatrix}
    A & B \\ 
    C & D \\
\end{pmatrix} \in \operatorname{SU}(p,q)$,
$Z\in \mathcal{E}$, and $f \in \mathscr{H}_\lambda$.
Note that 
$$
||1||_\lambda = \int_\mathcal{E} \det(I_q-Z^*Z)^{\lambda-(p+q)} < \infty
\iff
\lambda - (p+q) > - 1.
$$
Hence, we will restrict to
$\lambda > p+q - 1 $, in which case the constant function $1 \in \mathcal{H}_\lambda$, and 
consider the matrix coefficient 
\begin{IEEEeqnarray*}{rCl}
    \psi_\lambda(g) &=& \langle \pi_\lambda(g)1 , 1 \rangle_\lambda \\ 
&=&  (\pi_\lambda(g)1)(0) \\
&=& J_\lambda(g^{-1}, 0)^{-1} \\
&=&(\det D)^{-\lambda}, \quad \lambda \geqslant p+q, \quad
g^{-1} = \begin{pmatrix}
    A & B \\ C & D
\end{pmatrix}.
\end{IEEEeqnarray*}

\medskip

\subsubsection{Example 2:} 
Consider a realization of $\operatorname{Sp}(n,\R)$
defined by 
$$
G := \operatorname{Sp}(n,\mathbb{C})  \cap \operatorname{SU}(n,n)
= \{g\in \operatorname{SU}(n,n): g^t J g = J \}
$$
where 
$$
J = \begin{pmatrix}
    0 & I_n \\
    -I_n & 0  
\end{pmatrix}
$$
and $I_n$ is the $n\times n$ identity matrix.
We choose this realization because any 
$\begin{pmatrix}
    A & B \\ C & D
\end{pmatrix} \in \operatorname{SU}(n,n)$ has $\det D \neq 0$ \cite[p. 500]{Neeb}. 

The maximal compact subgroup $K$ is isomorphic to $\operatorname{U}(n)$
and
$$
G/K \simeq \mathcal{E} := \{Z\in M(n,\cc): Z^t = Z, ||Z||<1 \} 
$$
is a bounded symmetric domain
biholomorphic to the Siegel upper half plane
$$
\{X+iY: X,Y\in M(n,\R), X^t=X, Y^t=Y >0\}.
$$
The Siegel upper half plane is an example of a symmetric domain of tube type:
$$
G/K = V+i\Omega 
$$
where $\Omega$ is a symmetric cone.

The group $G$ has a holomorphic discrete series of unitary representations
$$
\pi_\lambda: G\to \operatorname{GL}(\mathscr{H}_\lambda) \subset L^2(G)
$$
with matrix coefficients
$$
\psi_\lambda(g) = \langle \pi_\lambda(g)1,1\rangle_\lambda.
$$
For
$$
g^{-1} = \begin{pmatrix}
    A & B \\ C & D
\end{pmatrix} \in \operatorname{Sp}(n,\mathbb{C})\cap \operatorname{SU}(n,n),
$$
$\det D \neq 0$ and 
$$
\psi_\lambda(g) = (\det D)^{-\lambda}, \quad \lambda =1,2,3,\ldots
$$

\begin{theorem} \label{G}
Let $G$ be a noncompact reductive group of nonexceptional type with maximal compact subgroup $K$ such that if $\mathfrak{c} = Z_\mathfrak{k}$ is the center in $\mathfrak{k}$, then 
    $$
Z_\mathfrak{g}(\mathfrak{c}) = \mathfrak{k}.
    $$
In the case $G \simeq \operatorname{Sp}(n,\R)$, we choose the realization 
$G =\operatorname{Sp}(n, \mathbb{C})\cap \operatorname{SU}(n,n)$.

Suppose $\mathscr{D} = \{ g_j \in G: 1 \leqslant j \leqslant n\}$ are such that if $g_j^{-1}= \left( \begin{smallmatrix}
   A_j & B_j \\ C_j & D_j
\end{smallmatrix} \right)$ then $\det D_j$ are distinct as $j$ traverses $1,\ldots,n$.
Let
$f:\mathscr{D}\to \cc$ be a nonconstant function.
For $g\in G$ with 
$ g^{-1} = \left( \begin{smallmatrix}
   * & *\\
   * & D
\end{smallmatrix} \right) \in G$, where $\det D \neq 0$, 
define $\psi_\lambda(g) = 
\det(D)^{-\lambda}$. 
Then there are $a_\lambda \in \cc$ such that 
\[
W_{f,\mathscr{D}}(g) = \sum_{\lambda} a_\lambda \psi_\lambda(g)
\]
is a holomorphic, square integrable Whitney extension of $f$. 
\end{theorem}

\begin{proof}
The noncompact Riemannian symmetric spaces $G/K$ of Hermitian type have been classified \cite[Ch. IX, \S 4.4]{Helgason}; there are four classical types with symmetry groups
$$ G \simeq 
\operatorname{SU}(p,q), \operatorname{Sp}(n,\R), \operatorname{SO}^*(2n), \operatorname{SO}(p,2)^\circ. $$
In the symplectic case, we take the realization $ G = \operatorname{Sp}(n,\mathbb{C})\cap \operatorname{SU}(n,n)$.
For any such $G$, each $g_j \in G$ has a natural block form
$$
g_j^{-1} = \begin{pmatrix}
    A_j & B_j \\
    C_j & D_j
\end{pmatrix}
$$
where  $\det D_j \neq 0$.
Since $D_j \neq D_i$ whenever $i<j$, section \ref{vandermondetrick} implies that there exists an exact Whitney extension $W_{f,\mathscr{D}} = \sum_\lambda a_\lambda \psi_\lambda : G \to \cc$ of $f: \mathscr{D}\to \cc$. 
The parameter $\lambda$ runs over a set of integers of cardinality $n$ and $W_{f,\mathscr{D}}$ is holomorphic and square integrable. 
\end{proof}

\bigskip

\bibliographystyle{amsplain} 
\bibliography{references.bib}

\nocite{*}

\end{document}